\documentclass[a4paper]{amsart}
\pdfoutput=1

\usepackage[all, hyperref]{paper_diening}
\renewcommand{\norm}[1]{\left\lVert#1\right\rVert}
\renewcommand{\abs}[1]{\left\lvert#1\right\rvert}

\usepackage[backend=biber,maxbibnames=5,sorting=nyt,maxalphanames=5,style=alphabetic,bibencoding=utf8,giveninits,url=false,isbn=false]{biblatex}
\AtBeginBibliography{\small}

\usepackage{amsmath,amsthm,amssymb}
\usepackage{epsfig}
\usepackage[utf8]{inputenc}
\usepackage[T1]{fontenc}
\usepackage{dsfont}
\usepackage{faktor}

\numberwithin{equation}{section}

\newcommand{\dd}{\,\mathrm{d}}

\providecommand{\seminormtmp}[2]{{#1[{#2}#1]}}
\providecommand{\seminorm}[1]{\seminormtmp{}{#1}}

\title[Convex integration for SQG with generic additive noise]{On convex integration solutions to the surface quasi-geostrophic equation driven by generic additive noise}
\author{Florian Bechtold, Theresa Lange, J\"orn Wichmann}
\email{\{fbechtold, tlange\}@math.uni-bielefeld.de,  Joern.Wichmann@monash.edu}

\begin{document}
\begin{abstract}
    We study the surface quasi-geostrophic equation driven by a generic additive noise process $W$. By means of convex integration techniques, we establish existence of weak solutions whenever the stochastic convolution $z$ associated with $W$ is well defined and fulfills certain regularity constraints. Quintessentially, we show that the so constructed solutions to the non-linear equation are controlled by $z$ in a linear fashion. This allows us to deduce further properties of the so constructed solutions, without relying on structural probabilistic properties such as Gaussianity, Markovianity or a martingale property of the underlying noise $W$. 
\end{abstract}
\maketitle

\noindent
\section{Introduction}
In this work, we are concerned with the construction of solutions to the surface quasi-geostrophic equation (SQG equation) on the two-dimensional torus $\mathbb{T}^2$
\begin{equation}
\begin{split}
    \label{sqg intro}
    \partial_t\theta+u\cdot \nabla \theta &=-\nu \Lambda^\gamma\theta+\Lambda^\delta \dot{W} \\
    u&=\nabla^{\perp}\Lambda^{-1}\theta=(-\partial_2 \Lambda^{-1}\theta, \partial_1 \Lambda^{-1} \theta)=(-\mathcal{R}_2 \theta, \mathcal{R}_1 \theta)=\mathcal{R}^\perp\theta
\end{split}
\end{equation}
where $\Lambda=(-\Delta)^{1/2}$, $\nu\geq 0$, $\gamma\in [0, 3/2)$, $\mathcal{R}$ is the pair of Riesz-transforms and $\dot{W}$ a generic noise process on a probability space $(\Omega, \mathcal{F}, \mathbb{P})$ for which the associated stochastic convolution 
\[
z_t:=\int_0^t e^{-\nu (t-s)\Lambda^\gamma}\Lambda^{\delta-1}{\rm d}W_s
\]
is well defined and satisfies certain integrability conditions to be made precise below. Towards this end, we adapt a convex integration scheme originating from \cite{MR4340931} and subsequently extended in \cite{hofmanovaconvex4, hofmanovaconvex6} working on the level of the transformed unknown $g=\Lambda^{-1}\theta-z$, which satisfies
\begin{align} \label{intro:Main}
\nabla \cdot \left( -\partial_t \mathcal{R} g + \nabla^{\perp} (g + z) \Lambda (g + z)\right) = \nabla \cdot \left(\nu \Lambda^{\gamma -1} \nabla g\right) \quad \text{ in } \Omega \times \mathbb{R} \times \mathbb{T}^2,
\end{align}

\begin{definition} \label{def:WeakSolution}
We call $g \in C_{\mathrm{loc}}\big(\mathbb{R}; \dot{H}^{1/2}(\mathbb{T}^2)\big)$ a weak solution to~\eqref{intro:Main} if
\begin{align} \label{eq:SolutionConcept}
\langle \nabla \cdot [ -\partial_t \mathcal{R} g], \xi \rangle - \langle \Lambda (g + z)  , \nabla^\perp(g+z) \cdot \nabla \xi  \rangle  = \langle \nabla \cdot[\nu \Lambda^{\gamma -1} \nabla g], \xi \rangle 
\end{align}
for all $\xi \in C^\infty_c ( \mathbb{R} \times \mathbb{T}^2 )$ and $\mathbb{P}$-a.s., where the non-linearity is understood in the sense of Lemma \ref{lem:Commutator}. 
\end{definition}
For the rigorous connection between \eqref{intro:Main} and \eqref{sqg intro}, we refer to subsection \ref{equation trafo intro} in the appendix. Our main result is the following. For the notational conventions employed, we refer to Section \ref{notation section}. 
\begin{theorem}
    Let $W$ be a stochastic process such that the expression 
    \[
    z_t=\int_0^t e^{-\nu(t-s)\Lambda^\gamma} \Lambda^{\delta-1}{\rm d}W_s
    \]
    is well defined and satisfies
    \begin{equation}
        \label{regularity z}
          \seminorm{z}_{L^{2m}_\omega C^\alpha_tC^{1+}_x}+\norm{z}_{L^{2m}_\omega C_{\rm loc}C^{1+\kappa}_x}<\infty
\end{equation}
    for some $m\geq 1, \alpha>0$ and $\kappa>1/2$, $\delta\in \mathbb{R}$ and $\gamma<3/2$. Then there exists a weak solution $g$ to \eqref{intro:Main} in the sense of Definition \ref{def:WeakSolution} satisfying
    \begin{equation}
            \norm{g}_{L^{2m}_\omega C_{\rm loc}B^{1/2}_{\infty, 1}}\leq C ( \seminorm{z}_{L^{2m}_\omega C^\alpha_tC^{1+}_x}+\norm{z}_{L^{2m}_\omega C_{\rm loc}C^{1+\kappa}_x}+1).
            \label{main bound}
    \end{equation}
\end{theorem}
Let us comment on the main novelties of the above result, to be discussed also later in the paper:
\begin{itemize}
    \item Despite the equation \eqref{intro:Main} being non-linear in $g$, we are able to construct solutions which are controlled by the forcing in a \textit{linear} fashion as expressed by \eqref{main bound}. 
    \item As a consequence, whenever quantified moment estimates on $z$ are available, they carry over to the solution $g$. For example, in the case of $W$ being a cylindrical Wiener process on $L^2(\mathbb{T}^2)$ (which in particular is Gaussian), this immediately implies sub-Gaussian moment estimates of the form 
    \[
    \mathbb{E}\left[ \exp{\epsilon \norm{g}_{C_{\rm loc}B^{1/2}_{\infty, 1}}^2}\right]<\infty
    \]
    for some $\epsilon>0$, unavailable in previous stochastic convex integration schemes. 
    \item We stress that in terms of the underlying stochastic process $W$, our Theorem merely requires \eqref{regularity z}. In particular, it is not necessary to assume much on the specific probabilistic structure of $W$ such as Markovianity, Gaussianity or the martingale property. Indeed, \eqref{regularity z} can be assured in rather general contexts, as we demonstrate in a dedicated examples section \ref{examples section}. This conveys a remarkable robustness of the convex integration scheme.
\end{itemize}

\subsection{Literature on SQG}
Let us begin by discussing results in the literature on the case of the deterministic SQG equations, i.e. \eqref{sqg intro} with $W=0$. As a special case of the quasi-geostrophic equations, they form a model of potential temperature in the context of geophysical fluid dynamics, describing the evolution of boundary temperature in a fast rotating stratified fluid. For more details on its physical interpretation, derivation and application in atmospheric sciences, meteorology and oceanography refer to \cite{pedlosky1982}, \cite{lapeyre2017}, \cite{constantin1994}, \cite{held1995}, where in particular \cite{constantin1994} illustrates an interesting analogy to the 3D Euler equations underlining the interest in studying SQG. In the deterministic case, though not exhaustive, one may highlight the following results: some works on well-posedness of SQG show global existence of weak solutions in case $\nu \geq 0, 0 < \gamma \leq 2$ in $L_t^{\infty}L_x^2$ for any initial data in $L^2(\mathbb{T}^2)$ \cite{resnick1996}, as well as in the case $\nu > 0, 0 < \gamma \leq 2$ in $L_t^{\infty}H_x^{-1/2}$ for initial conditions in $\dot{H}_x^{1/2}(\mathbb{R}^2)$, and of regularity $L_t^{\infty}L_x^p$ for initial conditions in $L_x^p(\mathbb{R}^2), p \geq 4/3$ \cite{marchand2008}. Some results and discussions on potential blow-up of solutions can be found in \cite{ohkitani1997}, \cite{cordoba1997}, \cite{cordoba1998}, \cite{constantin1999}, \cite{castro2010}, \cite{scott2011}. Furthermore, global existence of regular solutions was derived in \cite{foldes2021} via the existence of an invariant measure on $\mathbb{T}^2$ and for all initial conditions in the support of that measure. In terms of uniqueness, \cite{cordoba2018} showed uniqueness of SQG patches with moving boundary satisfying an arc-chord condition in the case $\nu =0$; however, for $\nu \geq 0, 0 < \gamma < 2 -\beta$, \cite{buckmaster2019} establish weak solutions $(-\Delta)^{-1/2}\theta \in C_t^{\sigma}C_x^{\beta}$ for $1/2 < \beta < 4/5, \sigma < \frac{\beta}{2-\beta}$, via convex integration rendering the solutions non-unique. Also see \cite{isett2021} for a direct approach. For the steady-state SQG equation, \cite{cheng2021} derive existence of stationary weak solutions employing a similar convex integration scheme yielding non-uniqueness of the solutions.\\
\\
There have been several results in the direction of "regularization by noise" obtained in connection with the SQG equations. Observe that \cite{foldes2021} derive an invariant measure for the deterministic case via invariant measures for SQG perturbed by a carefully chosen additive noise. Including random perturbations in general has shed new light on questions of well-posedness: in the cases of additive and linear multiplicative noise, \cite{rockner2014} and \cite{rockner2015} provide local existence and uniqueness (continuous in some $H^s$ on a smooth bounded domain \cite{rockner2014}, and in $L^p(\mathbb{T}^2), p \in (2,\infty)$ \cite{rockner2015}) as well as non-explosion and ergodicity results whereas in \cite{zhu2017} existence of a random attractor is analysed. We would furthermore like to mention the work \cite{forstner2021} on local well-posedness for space-time white noise via regularity structures, as well as a large deviation principle for small multiplicative noise and small time in \cite{liu2013}. Further note the regularization result in \cite{buckmaster2020} via random diffusions. 

\subsection{Stochastic convex integration results}
After the advent of deterministic convex integration results (for instance see \cite{Buckmaster20201} and references therein), it became natural to ask if the techniques developed in this context carry over to corresponding stochastic problems in order to establish non-uniqueness results. With this background, Hofmanová, Zhu and Zhu were able to adapt convex integration methods to hydrodynamic equations with noise. Their first papers \cite{hofmanovaconvex1, hofmanovaconvex3, hofmanovaconvex2} did so by working with stopping time techniques, which has the advantage of avoiding moment estimates along the convex integration iteration. For a gentle overview of these results, refer also to \cite{Hofmanovoverview}. In this context, let us also mention the works of Yamazaki \cite{Yamazaki2022, Yamazaki20222, Yamazaki2023} and Schenke and Rehmeier \cite{Rehmeier2023} where in particular \cite{yamazaki20221} treats the case of SQG equations with additive noise deriving existence and non-uniqueness in law of solutions of regularity $C_tC_x^{1/2+\iota} \cap C_t^{\eta}C_x, \iota \in\left(0,1/2-\gamma/3\right), \eta \in (0,1/3]$. Establishing moment estimates poses a problem in all these works, as convex integration schemes typically involve super-linear estimates that need to be controlled through careful parameter-tuning. In particular, as the estimate of a given moment in iteration $n+1$ thus requires higher and higher moment estimates on the previous levels, one is unable to close moment estimates. \\
\\
In \cite{chen2022}, this issue was for the first time overcome by a moment-dependent parameter-tuning that allows to control the blow-up of higher and higher moments along the iteration scheme, exploiting the quantified moment behaviour of the underlying Gaussian noise. While this opened the door for the establishment of stationary solutions through a Krylov-Boguliobov type argument in \cite{hofmanovaconvex5}, the drawback of the approach consists in the fact that one can not simultaneously study different moments of a specific convex integration, as the scheme itself depends on the moment one intends to inverstigate in the beginning. Also, as the tuning depends on the exact moment behaviour of the underlying noise, extensions to different noises are not obvious. \\
\\
Inspired by \cite{cheng2021}, Hofmanov\'a et al. study moment-based convex integration in the context of the SQG equations: for the case of only space dependent noise \cite{hofmanovaconvex4} and recently space-time white noise \cite{hofmanovaconvex6}, they obtain infinitely many non-Gaussian weak solutions (of regularity $L^p_{\rm loc}(0,\infty);B_{\infty,1}^{-1/2}) \cap C_b([0,\infty);B_{\infty,1}^{-1/2-})\cap C_b^1([0,\infty);B_{\infty,1}^{-3/2-})$ for all $p \in [1,\infty)$ \cite{hofmanovaconvex4}, and of regularity $L^p_{\rm loc}([0,\infty);B_{\infty,1}^{-1/2})\cap C([0,\infty),B_{\infty,1}^{-1/2-})\cap C^{1/2-}([0,\infty),B_{\infty,1}^{-1-})$ for all $p \in [1,\infty)$ \cite{hofmanovaconvex6}, respectively) as well as infinitely many stationary (ergodic) solutions. While one main focus of \cite{hofmanovaconvex6} consists in demonstrating that convex integration schemes can serve to establish solutions to supercritical SPDEs inaccessible by tools such as regularity structures or paracontrolled distributions, the authors were able to also show that the convex integration scheme can be set up independently of a given moment fixed in the beginning. This allows them to show that arbitrary moments of the constructed convex integration solutions exist, provided this is the case for the initial condition. Further quantified moment bounds, however, are lacking. In particular, the questions of existence of exponential moments is left open and the consideration of more general noises is lacking.

\subsection{Structure of the paper}
After introducing some basic notation and preliminaries in Section \ref{notation section}, we proceed to give the convex integration construction and thus the proof of our main theorem in Section \ref{convex int section}. We then provide some concrete examples of noises to which our result is applicable, i.e. for which we explicitly verify the condition \eqref{regularity z}. We treat one class of general noises in the Young integration setting by means of the Sewing Lemma and another class of noises by means of more classical stochastic calculus arguments, in neither assuming Gaussianity, the Markov nor martingale property. We finally conclude by discussing some immediate consequences of our main result which can be obtained by a slight adaption of arguments in \cite{hofmanovaconvex4, hofmanovaconvex6}.


\section{Notation and preliminaries}
\label{notation section}
We write $f \lesssim g$ for two non-negative quantities $f$ and $g$ if $f$ is bounded by $g$ up to a multiplicative universal constant. Accordingly we define $\gtrsim$ and $\eqsim$. For two Banach spaces $X$ and $Y$ we write $X \hookrightarrow Y$ if the inclusion is continuous. The duality pairing between $C^\infty_c(\mathbb{R} \times \mathbb{T}^2)$ (compactly supported and smooth functions) and $\big(C^\infty_c(\mathbb{R} \times \mathbb{T}^2) \big)'$ is denoted by $\langle \cdot, \cdot\rangle$. The convolution of two functions in time is denoted by $*_t$, i.e., $f *_t g (s) = \int_{\mathbb{R}} f(s-r) g(r) \dd r$.

We denote by $C^{n+\kappa}_x$, $n \in \mathbb{N}_0 = \mathbb{N} \cup \{0\}$ and $\kappa \in [0,1)$, the space of $n$-times continuously differentiable functions with $\kappa$-H\"older continuous $n$-th derivative equipped with the norm 
\begin{align*}
    \norm{g}_{C^{n+\kappa}_x} =  \sum_{\abs{k} \leq n} \sup_{x \in \mathbb{T}^2} \abs{\partial^{k} g(x)}  + \sum_{\abs{k} = n} \sup_{x\neq y \in \mathbb{T}^2} \frac{\abs{\partial^{k} g(x) - \partial^{k} g(y)}}{\abs{x-y}^\kappa},
\end{align*}
where $\partial^k = \partial^{k_1} \partial^{k_2}$ for a multi index $k \in \mathbb{N}_0^2$.

Let $E$ be a Banach space and $\alpha \in [0,1)$. We define the Banach space of ($\alpha$-H\"older) continuous $E$-valued functions on $[t, t+1]$ by $C_t E$ and $C_t^\alpha E$, respectively. We equip them with the norms
\begin{align*}
\norm{g}_{C_t E} &:=  \sup_{s \in [t,t+1]} \norm{g(s)}_E, 
\end{align*}
and
\begin{align*}
\seminorm{g}_{C_t^\alpha E} &:= \sup_{s\neq s' \in [t,t+1]} \frac{\norm{g(s) - g(s')}_E}{\abs{s-s'}^\alpha}, \\
\norm{g}_{C_t^\alpha E} &:= \norm{g}_{C_t E} + \seminorm{g}_{C_t^\alpha E},
\end{align*}
respectively. Moreover, for any $T>0$, we define similarly $C_{[0, T]}$ and $C^\alpha_{[0, T]}$ via 
\begin{align*}
\norm{g}_{C_{[0, T]} E} &:=  \sup_{s \in [0,T]} \norm{g(s)}_E, 
\end{align*}
and
\begin{align*}
\seminorm{g}_{C_{[0, T]}^\alpha E} &:= \sup_{s\neq s' \in [0, T]} \frac{\norm{g(s) - g(s')}_E}{\abs{s-s'}^\alpha}, \\
\norm{g}_{C_{[0, T]}^\alpha E} &:= \norm{g}_{C_{[0, T]} E} + \seminorm{g}_{C_t^\alpha E},
\end{align*}

Let $(\Omega, \mathcal{F}, \mathbb{P})$ be a probability space. The space of strongly measurable $L^m$-integrable, $m \in [1,\infty]$, $E$-valued random variables is denoted by $L^m_\omega E$ equipped with the usual $L^m$-norm. Let $L^m_\omega C_{\mathrm{loc}} E$, respectively $L^m_\omega C^\alpha_{\mathrm{loc}} E$, denote the space of $L^{m}$-integrable locally ($\alpha$-H\"older) continuous $E$-valued functions equipped with the uniform time-localized (semi)-norms, $m \in [1,\infty)$, $\alpha \in (0,1)$, defined by
\begin{align*}
\norm{g}_{L^m_\omega C_{\mathrm{loc}}  E}^m &:= \sup_{t \in \mathbb{R}} \mathbb{E} \left[ \norm{g}_{C_t E}^m \right],  \\
\seminorm{g}_{L^m_\omega C_{\mathrm{loc}} ^\alpha E}^m &:= \sup_{t \in \mathbb{R}} \mathbb{E} \left[ \seminorm{g}_{C_t^\alpha E}^m \right], \\
\norm{g}_{L^{m}_\omega C_{\mathrm{loc}}^\alpha E}^m &:= \norm{g}_{L^m_\omega C_{\mathrm{loc}}  E}^m + \seminorm{g}_{L^m_\omega C_{\mathrm{loc}}^\alpha E}^m .
\end{align*}

The Fourier transform on $\mathbb{T}^2$ is defined by
\begin{align*}
 \mathbb{Z}^2 \ni k \mapsto   \widehat{f}(k) := \int_{\mathbb{T}^2}e^{ix \cdot k} f(x) \dd x.
\end{align*}

For $s \in \mathbb{R}$, we define $\dot{H}^{1/2} = \{ f: \int_{\mathbb{T}^2} f(x) \dd x = 0, \norm{\Lambda^s f}_{L^2(\mathbb{T}^2)} < \infty \}$ equipped with the norm 
\begin{align*}
    \norm{f}_{\dot{H}^{s}}^2 := \sum_{0 \neq k \in \mathbb{Z}^2} \abs{k}^{2s} \abs{\widehat{f}(k)}^2. 
\end{align*}

Let $(\Delta_j)_{j \geq -1}$ denote Littlewood-Paley blocks corresponding to a dyadic partition of unity. We define Besov spaces $B^{s}_{p,q}$, $s \in \mathbb{R}$, $p,q \in [1,\infty]$, as the completion of smooth functions with respect to the norm
\begin{align*}
    \norm{g}_{B^{s}_{p,q}} := \norm{ \{ 2^{j\alpha} \norm {\Delta_j u}_{L^p} \}_j }_{\ell^q}.
\end{align*}
Here $\norm{\cdot}_{\ell^q}$ denotes the usual $\ell^q$-sequence norm.

Let $\psi \in C^\infty_c(\mathbb{R}^2)$ satisfy $\psi(k) = 0$ for $\abs{k} \geq 1$ and $\psi(k) = 1$ for $\abs{k} \leq 1/2$. Define the frequency truncation $P_{\leq \lambda}$ for $\lambda > 0$ by 
\begin{align*}
    \widehat{P_{\leq \lambda}g} (k) = \psi\left( \frac{k}{\lambda} \right) \widehat{g}(k).
\end{align*}
Notice that $\norm{P_{\leq \lambda}g}_{L^\infty} \lesssim \norm{g}_{L^\infty}$ for all $g \in L^\infty$ and $\lambda \geq 1$. For more details on Besov spaces and the stability of frequency truncation we refer to~\cite[Section 1.3]{MR2848761}. We write $g \in B^{s-}_{p,q}$ if $g \in B^{s-\varepsilon}_{p,q}$ for any $\epsilon >0$ and $g \in B^{s+}_{p,q}$ if there exists $\varepsilon > 0$ such that $g \in B^{s+\varepsilon}_{p,q}$. 

For any vector fields $v,w$ we write $ v \, \mathring{\approx} \, w$ if there exists a smooth scalar function~$p$ such that $v = w + \nabla^\perp p$. 

We define, similarly to~\cite[(2.7)]{MR4340931}, the Riesz-type transforms $\mathcal{R}_j^o$, $j = 1,2$, by 
\begin{align*}
    \widehat{\mathcal{R}_1^o}(k) := \frac{25(k_2^2 - k_1^2)}{12 \abs{k}^2}, \hspace{2em} \widehat{\mathcal{R}_2^o}(k) := \frac{7(k_2^2 - k_1^2)}{12 \abs{k}^2} + \frac{4 k_1 k_2}{\abs{k}^2}
\end{align*}
and 
\begin{align*}
    \norm{q}_X := \norm{q}_{L^\infty} + \sum_{j=1}^2 \norm{\mathcal{R}_j^o q}_{L^\infty}.
\end{align*}

\section{Convex integration}
\label{convex int section}
We want to construct a solution as the limit of an iterative algorithm, i.e.,
\begin{align*}
g = \lim_{n\to \infty} g_{\leq n}, \hspace{3em} g_{\leq n} = g_{\leq n-1} + \tilde{g}_n,
\end{align*}
where $\tilde{g}_n$ are constructed in an explicit manner.
 
At first, the iterates~$g_{\leq n}$ do not need to satisfy~\eqref{intro:Main} exactly, but are allowed to carry a residual~$\tilde{q}_n$, i.e.,
\begin{align*}
-\partial_t \mathcal{R} g_{\leq n} + \nabla^{\perp} (g_{\leq n} + z) \Lambda (g_{\leq n} + z)-\nu \Lambda^{\gamma -1} \nabla g_{\leq n} \mathring{\approx}  \nabla \tilde{q}_n.
\end{align*}
The main difficulty (that enables convex integration) is to find an explicit update rule $\tilde{g}_n$ that corrects $g_{\leq n}$ in a way such that the magnitude of the correction is proportional to the residual. Ultimately, we hope for concentration of $g_{\leq N}$ on $g$ and a vanishing residual. 

Crucially, Cheng, Kwon and Li~\cite[Section~2.1]{MR4340931} found suitable building blocks to correct the iterates adapted to the SQG equation. They were modified to the stochastic setting by Hofmanová, Zhu and Zhu in~\cite{hofmanovaconvex4} (see also~\cite{hofmanovaconvex6}). 

We split our convex integration strategy into the following steps:
\begin{enumerate}
    \item construction and properties of explicit building blocks,
    \item propagation of residual magnitudes,
    \item an example that allows to close the argumentation.
\end{enumerate}

\subsection{Building blocks}
Let $(\lambda_n, r_n, \ell_n, \mu_n) \in (0,\infty)^4$ be control parameters. One can think of $\lambda_n (\nearrow \infty)$ measuring blow up and frequency truncation, $r_n (\searrow 0)$ concentration, $\ell_n (\searrow 0)$ convolution and $\mu_n (\nearrow \infty)$ frequency truncation on an intermediate scale, respectively. We will keep track on the requirements of the control parameters in order to close the argumentation. A possible choice can be found in~\ref{sec:ExampleControls}.

\subsubsection{Ansatz}
Let $z_n = P_{\leq \lambda_n} z$. We start with the Ansatz
\begin{subequations} \label{def:Init}
\begin{align} \label{def:InitG}
g_{\leq 0} &= 0, \\ \label{def:InitResidual}
\nabla q_0 &\,\, \mathring{\approx} \, \nabla^\perp z_0 \Lambda z_0,
\end{align}
\end{subequations}
and update recursively, for $n\geq 1$,
\begin{subequations} \label{def:Iteration}
\begin{align} \label{def:gLeQn}
g_{\leq n} &= g_{\ell_n} + g_{n}, \\ \label{eq:Residual}
 \nabla q_{n} &\,\, \mathring{\approx} -\partial_t \mathcal{R} g_{\leq n} + \nabla^{\perp} (g_{\leq n} + z_n) \Lambda (g_{\leq n} + z_n)- \nu \Lambda^{\gamma -1} \nabla g_{\leq n}.
\end{align}
\end{subequations}
Here the building blocks are defined by
\begin{subequations}
\begin{align} \label{def:Block01}
	g_{n} (t,x) &= \sum_{j=1}^2 a_{j, n}(t,x) \cos(5 \lambda_{n} l_j \cdot x),\\ \label{def:Block02}
	a_{j,n} &= 2\sqrt{\frac{\chi_{\ell_n}}{5 \lambda_{n}}}P_{\leq \mu_{n}}\sqrt{C_0 + \mathcal{R}_j^o \frac{q_{\ell_{n}}}{\chi_{\ell_n}}},
\end{align}
\end{subequations}
with 
\begin{subequations} \label{eq:MollifySteps}
\begin{align} 
	g_{\ell_n} &= g_{\leq n-1} *_t \psi_{\ell_{n}}, \\
    q_{\ell_n} &= q_{n-1} *_t \psi_{\ell_n}, \\
    \chi_{\ell_n} &= \norm{q_{n-1}}_X *_t \psi_{\ell_{n}} + r_{n-1},
\end{align}
\end{subequations}
where $\psi_\varepsilon (t)$ denotes a one-sided time mollifier, e.g., $ \psi_\varepsilon (t) = \varepsilon^{-1} \psi_0( \varepsilon^{-1} (t+2)) $ for some $\psi_0 \in C^\infty(\mathbb{R};[0,1])$, $\psi_0(t) = 1$ for $\abs{t} \leq 1$, $\psi_0(t) = 0$ for $\abs{t} \geq 2$ and $\int_{\mathbb{R}} \psi_0(t) \dd t = 1$. The free parameter $C_0 \geq 2$ can be used to derive non-uniqueness results as done in e.g.~\cite{MR4340931,hofmanovaconvex4,hofmanovaconvex6}.

Next, we derive bounds of the building blocks in terms of the control parameters and the residual. Let $m \geq 1$ be the probabilistic integrability, $\alpha > 0$ be the time H\"older regularity and $\kappa > 1/2$ be the space H\"older-regularity. Set
\begin{align} \label{def:ConstantPhi}
C_z := \norm{z}_{L^{2m}_\omega C_{\mathrm{loc}} C^{1+\kappa}_x} + \seminorm{z }_{L^{2m}_\omega C^\alpha_{\mathrm{loc}} C^{1+}_x}.
\end{align}
It turns out that $C_z < \infty$ is sufficient for finding suitable control parameters.

\subsubsection{Properties}
Similarly to~\cite[Remark~6.1]{hofmanovaconvex6} we find that $g_n$ is frequency localized in the annulus~$\{k: 5 \lambda_n - \mu_n \leq \abs{k} \leq 5 \lambda_n + \mu_n\}$. In particular, $g_n$ is mean-free. Since the mollifier in~\eqref{eq:MollifySteps} is one-sided, $g_n$ is adapted to the filtration $\mathcal{F}_t = \sigma( z_s | s \leq t ) $. We derive some further preparatory estimates for $g_{n}$. 

\subsubsection{Continuous scale}
Since $g_n$ is frequency localized
\begin{align*}
\norm{g_n(t)}_{B^{1/2}_{\infty,1}} &= \sum_{j \geq -1} 2^{j/2} \norm{\Delta_j g_n(t)}_{L^\infty} = \sum_{j:5 \lambda_n - \mu_n \leq 2^j \leq 5 \lambda_n + \mu_n} 2^{j/2} \norm{\Delta_j g_n(t)}_{L^\infty}. 
\end{align*}

In order to simplify the estimate we restrict to
\begin{align} \tag{R1} \label{R1}
\mu_n \leq \lambda_n.
\end{align}
Then $\{k: 5 \lambda_n - \mu_n \leq \abs{k} \leq 5 \lambda_n + \mu_n\} \subset \{k: 4 \lambda_n \leq \abs{k} \leq 6 \lambda_n\}$ and 
\begin{align*}
\sum_{j:5 \lambda_n - \mu_n \leq 2^j \leq 5 \lambda_n + \mu_n} 2^{j/2} \norm{\Delta_j g_n(t)}_{L^\infty} &\leq  (6 \lambda_n)^{1/2}\sum_{j: 4 \lambda_n \leq 2^j \leq 6 \lambda_n} \norm{\Delta_j g_n(t)}_{L^\infty} \\
&\lesssim \lambda_n^{1/2} \norm{g_n(t)}_{L^\infty}.
\end{align*}
Here we used that 
\begin{align*}
\sum_{j: 4 \lambda_n \leq 2^j \leq 6 \lambda_n}1 \lesssim 1,
\end{align*}
since $\# \{ j: 4\lambda_n \leq 2^j \leq 6 \lambda_n \} = \# \{ j: 2 + \log_2(\lambda_n) \leq j \leq \log_2(6) + \log_2(\lambda_n)\}$ and $C^0 \hookrightarrow B^0_{\infty,\infty}$ (see for example \cite[p.89]{Triebel1983}). Moreover, due to stability of the frequency truncation and $\norm{q_{\ell_n}}_X \leq \chi_{\ell_n}$,
\begin{align*}
 \norm{g_n(t)}_{L^\infty} &\leq \sum_{j=1}^2 \norm{a_{n,j}(t)}_{L^\infty} \\
&\lesssim \sqrt{\frac{\chi_{\ell_n}(t)}{5 \lambda_{n}}} \sum_{j=1}^2  \norm{C_0 + \mathcal{R}_j^o \frac{q_{\ell_n}(t)}{\chi_{\ell_n}(t)}}_{L^\infty}^{1/2} \\
&\lesssim \sqrt{\frac{\chi_{\ell_n}(t)}{ \lambda_{n}}}  \left(C_0+  \frac{\norm{ q_{\ell_n}(t)}_{X}}{\chi_{\ell_n}(t)} \right)^{1/2} \\
&\lesssim  \sqrt{\frac{\chi_{\ell_n}(t)}{\lambda_{n}}}  \left(C_0+ 1 \right)^{1/2}.
\end{align*}
Thus,
\begin{align*} 
\norm{g_n(t)}_{B^{1/2}_{\infty,1}}  \lesssim  \sqrt {(C_0+ 1) \chi_{\ell_n}(t) }.
\end{align*}
Take the supremum over $t \in [s,s+1]$, $L^{2m}_\omega$-norm and supremum in $s \in \mathbb{R}$,
\begin{align*}
\norm{g_n}_{L^{2m}_\omega C_{\mathrm{loc}} B^{1/2}_{\infty,1}}  \lesssim (C_0+ 1) \sup_{s \in \mathbb{R}} \left( \mathbb{E}\left[ \sup_{t \in [s,s+1]} \chi_{\ell_n}(t)^m  \right] \right)^{1/2m}.
\end{align*}
We restrict our setting to
\begin{align} \tag{R2} \label{R2}
\ell_n \leq 1.
\end{align}
Then $[s-\ell_n,s+1+\ell_n] \subset [s-1,s+2] = \cup_{k=0}^2 [s-1+k,s+k]$ and thus,
\begin{align} \label{eq:EstimateChi}
\sup_{t \in [s,s+1]} \chi_{\ell_n}(t) \leq \sup_{t \in [s - 1,s+2 ]} \norm{q_{n-1}(t)}_X + r_{n-1}.
\end{align}
Therefore,
\begin{align}\label{eq:EstimateBuildingBlock}
\norm{g_n}_{L^{2m}_\omega C_{\mathrm{loc}} B^{1/2}_{\infty,1}} \lesssim \sqrt{(C_0 + 1)( \norm{q_{n-1}}_{L^m_\omega C_{\mathrm{loc}} X} + r_{n-1})}.
\end{align}
Similarly,
\begin{align}\label{eq:EstimateBuildingBlock02}
\norm{g_n}_{L^{2m}_\omega C_{\mathrm{loc}} C^\vartheta_x} \lesssim \lambda_n^{\vartheta - 1/2} \sqrt{(C_0 + 1)( \norm{q_{n-1}}_{L^m_\omega C_{\mathrm{loc}} X} + r_{n-1})}.
\end{align}

\begin{remark}
The estimate~\eqref{eq:EstimateBuildingBlock} shows that the building blocks decay as a function of the residual. Moreover,~\eqref{eq:EstimateBuildingBlock02} allows for stronger spatial scales provided the residual decays sufficiently fast.
\end{remark}

\subsubsection{H\"older scale}
Let $s \neq s' \in [t, t+1]$. Analogously to previous section, one finds for increments
\begin{align*}
\norm{g_{n}(s) - g_n(s')}_{B^{1/2}_{\infty,1}} \lesssim \lambda_n^{1/2} \sum_{j=1}^2 \norm{a_{j,n}(s) - a_{j,n}(s')}_{L^\infty}.
\end{align*}
Using $\abs{\sqrt{a} - \sqrt{b}} \leq \sqrt{\abs{a-b}}$ and the stability of frequency truncation
\begin{align*}
\norm{a_{j,n}(s) - a_{j,n}(s')}_{L^\infty} &\lesssim \lambda_n^{-1/2}\sqrt{\abs{\chi_{\ell_n}(s) - \chi_{\ell_n}(s')}} \norm{\sqrt{C_0 + \mathcal{R}_j^o \frac{q_{\ell_{n}}(s)}{\chi_{\ell_n}(s)}} }_{L^\infty} \\
&\quad + \lambda_n^{-1/2} \sqrt{\chi_{\ell_n}(s')}\norm{\mathcal{R}_j^o \frac{q_{\ell_{n}}(s)}{\chi_{\ell_n}(s)} -\mathcal{R}_j^o \frac{q_{\ell_{n}}(s')}{\chi_{\ell_n}(s')}}_{L^\infty}^{1/2}.
\end{align*}
Notice that $\norm{\sqrt{C_0 + \mathcal{R}_j^o \frac{q_{\ell_{n}}(s)}{\chi_{\ell_n}(s)}} }_{L^\infty} \leq \sqrt{C_0 + 1}$ by definition of~$\chi_{\ell_n}$. Moreover, recall~\eqref{R2} and $\abs{s-s'}\leq 1$. Now, Young's inequality and properties of convolution ensure
\begin{align*}
&\abs{\chi_{\ell_n}(s) - \chi_{\ell_n}(s')} = \abs{\int \norm{q_{n-1}(\nu)}_X (\psi_{\ell_n}(s-\nu) - \psi_{\ell_n}(s'-\nu)) \dd \nu } \\
&\quad \leq \sup_{ \nu \in [s-\ell_n,s+\ell_n]\cup [s'-\ell_n,s'+\ell_n]} \norm{q_{n-1}(\nu)}_X  \int  \abs{\psi_{\ell_n}(s-\nu) - \psi_{\ell_n}(s'-\nu)} \dd \nu \\
&\quad \lesssim \sup_{\nu \in [t-1,t+3]} \norm{q_{n-1}(\nu)}_X \abs{s-s'}^{2\alpha} \ell_n^{-2\alpha}.
\end{align*}
Additionally,
\begin{align*}
&\norm{\mathcal{R}_j^o \frac{q_{\ell_{n}}(s)}{\chi_{\ell_n}(s)} -\mathcal{R}_j^o \frac{q_{\ell_{n}}(s')}{\chi_{\ell_n}(s')}}_{L^\infty} =  \frac{\norm{ \chi_{\ell_n}(s')  \mathcal{R}_j^o q_{\ell_{n}}(s)-\chi_{\ell_n}(s)\mathcal{R}_j^o q_{\ell_{n}}(s')}_{L^\infty}}{\chi_{\ell_n}(s) \chi_{\ell_n}(s')} \\
&\hspace{2em} \leq \frac{\abs{\chi_{\ell_n}(s') -\chi_{\ell_n}(s)}}{\chi_{\ell_n}(s')} \frac{\norm{\mathcal{R}_j^o q_{\ell_{n}}(s)}_{L^\infty}}{\chi_{\ell_n}(s) }  + \frac{\norm{ \mathcal{R}_j^o q_{\ell_{n}}(s)-\mathcal{R}_j^o q_{\ell_{n}}(s')}_{L^\infty}}{\chi_{\ell_n}(s')}\\
&\hspace{2em} \leq \frac{\abs{\chi_{\ell_n}(s') -\chi_{\ell_n}(s)}}{\chi_{\ell_n}(s')} + \frac{\norm{ q_{\ell_{n}}(s)- q_{\ell_{n}}(s')}_{X}}{\chi_{\ell_n}(s')}.
\end{align*}
Similarly to the estimate for increments of~$\chi_{\ell_n}$,
\begin{align*}
\norm{ q_{\ell_{n}}(s)- q_{\ell_{n}}(s')}_{X} &\leq \norm{\int q_{n-1}(\nu) (\psi_{\ell_n}(s-\nu) - \psi_{\ell_n}(s'-\nu)) \dd \nu }_X \\
&\lesssim  \sup_{\nu \in [t-1,t+3]} \norm{q_{n-1}(\nu)}_X \abs{s-s'}^{2\alpha} \ell_n^{-2\alpha}.
\end{align*}
Collecting all estimates shows 
\begin{align*}
\norm{a_{j,n}(s) - a_{j,n}(s')}_{L^\infty} \lesssim \left( (C_0 + 1) \lambda_n^{-1} \ell_n^{-2\alpha} \abs{s-s'}^{2\alpha} \sup_{\nu \in [t-1,t+3]} \norm{q_{n-1}(\nu)}_X \right)^{1/2}
\end{align*}
and
\begin{align}\label{eq:TimeSpace_Besov}
\seminorm{g_n}_{L^{2m}_\omega C^{\alpha}_{\mathrm{loc}} B^{1/2}_{\infty,1}} \lesssim \ell_n^{-\alpha} \sqrt{(C_0 + 1)  \norm{q_{n-1}}_{L^m_\omega C_{\mathrm{loc}} X} }.
\end{align}
Similarly,
\begin{align} \label{eq:TimeSpace_Hoelder}
\seminorm{g_n}_{L^{2m}_\omega C^{\alpha}_{\mathrm{loc}} C^\beta_x} \lesssim \lambda_{n}^{\beta - 1/2} \ell_n^{-\alpha} \sqrt{(C_0 + 1)  \norm{q_{n-1}}_{L^m_\omega C_{\mathrm{loc}} X}}.
\end{align}

\subsubsection{$C^1_t$ scale}
Take the time derivative in~\eqref{def:Block02}
\begin{align*}
	\partial_t a_{j,n} &= \frac{2}{\sqrt{5 \lambda_{n}}}\left(\frac{\partial_t\chi_{\ell_n}}{\sqrt{\chi_{\ell_n}}} P_{\leq \mu_{n}} \sqrt{C_0 + \mathcal{R}_j^o\frac{q_{\ell_{n}}}{\chi_{\ell_n}}} \right.\\
	&\left. + \sqrt{\chi_{\ell_n}}P_{\leq \mu_{n}} \left(C_0 + \mathcal{R}_j^o \frac{q_{\ell_{n}}}{\chi_{\ell_n}}\right)^{-\frac{1}{2}} \mathcal{R}_j^o\left(\frac{\partial_t q_{\ell_{n}}}{\chi_{\ell_n}} - \frac{q_{\ell_{n}}\partial_t\chi_{\ell_n}}{\chi_{\ell_n}^2}\right)\right).
\end{align*}
Notice, due to Jensen's inequality,
\begin{align*}
\abs{\mathcal{R}_j^0 q_{\ell_n}} \leq \norm{\mathcal{R}_j^0 q_{\ell_n}}_{L^\infty} \leq  \norm{q_{\ell_n}}_{X} \leq \norm{q_{n-1}}_X *_t \psi_{\ell_n} \leq \chi_{\ell_n}.
\end{align*}
Thus, $\mathcal{R}_j^0 q_{\ell_n} \chi_{\ell_n}^{-1}  \in [-1,1]$ and 
\begin{align*}
\norm{\sqrt{C_0 + \mathcal{R}_j^o\frac{q_{\ell_{n}}}{\chi_{\ell_n}}}}_{L^\infty} \leq \sqrt{C_0 + 1}, \hspace{2em} \norm{\left( C_0 + \mathcal{R}_j^o\frac{q_{\ell_{n}}}{\chi_{\ell_n}} \right)^{-1/2}}_{L^\infty} \leq \frac{1}{\sqrt{C_0 -1}}.
\end{align*}
This implies
\begin{align*}
	\|\partial_t a_{j,n}\|_{L^{\infty}} &\lesssim \lambda_{n}^{-\frac{1}{2}} \left(\frac{|\partial_t\chi_{\ell_n}|}{\sqrt{\chi_{\ell_n}}} \left( \sqrt{C_0 + 1} + \frac{1}{\sqrt{C_0 - 1}} \right) + \frac{1}{\sqrt{C_0 - 1}} \frac{\norm{\partial_t q_{\ell_n}}_X}{\sqrt{\chi_{\ell_n}}}  \right).
\end{align*}
Additionally, for $t \in [s,s+1]$
\begin{align*}
\abs{\partial_t \chi_{\ell_n}(t)} \vee \norm{\partial_t q_{\ell_n}(t)}_X &\leq \norm{q_{n-1}}_X *_t \abs{\partial_t \psi_{\ell_n}} \\
&\leq \sup_{t \in [s - 1, s+2]} \norm{q_{n-1}(t)}_X \norm{\partial_t \psi_{\ell_n}}_{L^1_t} \\
&\lesssim \sup_{t \in [s - 1, s+2]} \norm{q_{n-1}(t)}_X \ell_n^{-1}.
\end{align*}
All together, using $\chi_{\ell_n} \geq r_{n-1}$,
\begin{align} \label{eq:TimeDerivative}
\norm{\partial_t g_n}_{L^{m}_\omega C_{\mathrm{loc}} L^\infty} \lesssim \lambda_{n}^{-1/2} r_{n-1}^{-1/2} \ell_n^{-1} \norm{q_{n-1}}_{L^{m}_\omega C_{\mathrm{loc}} X} \left( \sqrt{C_0 + 1} + \frac{1}{\sqrt{C_0 - 1}} \right). 
\end{align}

\subsection{The residual in terms of control parameters}
The following proposition provides a quantified bound on the magnitude of the new residual in terms of the control parameters provided that the old residuals are controlled by the concentration parameter~$r_n$.

\begin{proposition}\label{prop:ResidualPropagation}
Let us assume that for $n \leq N$ the following conditions are satisfied: 
\begin{enumerate}
    \item (condition on control parameters)~\eqref{R1},~\eqref{R2} and
    \begin{align}\tag{R3} \label{R3}
\lambda_n \in \mathbb{N}, \hspace{2em} 12 \lambda_n \leq 4 \mu_{n+1}, \hspace{2em} 10 \leq \mu_n, \hspace{2em} 48 \lambda_n \leq \lambda_{n+1};
\end{align}
\item (control on past residuals)
\end{enumerate}
\begin{align}\label{ass:Induction} 
\norm{q_{n}}_{L^m_\omega C_{\mathrm{loc}} X} \leq r_n.
\end{align}

Then there exists a constant $C_{\mathrm{it}} \geq 1$ (independent of $N$) such that
\begin{align} \label{eq:ResidualPropagation}
\begin{aligned}
\norm{q_{N+1}}_{L^m_\omega C_{\mathrm{loc}} X} &\leq C_{\mathrm{it}}\big( \mathfrak{E}_{N+1}^{\mathrm{miss}} + \mathfrak{E}_{N+1}^{\mathrm{com}} + \mathfrak{E}_{N+1}^{\mathrm{time}} + \mathfrak{E}_{N+1}^{\mathrm{dis}} + \mathfrak{E}_{N+1}^{\mathrm{trans}} + \mathfrak{E}_{N+1}^{\mathrm{sto}}\big),
\end{aligned}
\end{align}
where 
\begin{subequations} \label{Error:Global}
    \begin{align} \label{Error:Miss}
        \mathfrak{E}_{N+1}^{\mathrm{miss}} &=  \left\{ \log \mu_{N+1} \left( \frac{\lambda_N}{\mu_{N+1}}  + \frac{\mu_{N+1}}{\lambda_{N+1}} \right)^2 + \frac{\lambda_N}{\lambda_{N+1}} \right\}r_{N}, \\ \label{Error:Com}
        \mathfrak{E}_{N+1}^{\mathrm{com}} &= \log(\lambda_{N+1}) \lambda_N \ell_{N+1}^{\alpha}\mathfrak{S}_{N,\alpha} \mathfrak{S}_{N,0} \\ \nonumber
&\quad+  \log(\lambda_{N+1}) \lambda_N^{1/2} \left( \ell_{N+1}^{\alpha}\mathfrak{S}_{N,\alpha} + \lambda_N^{-\kappa} \mathfrak{S}_{N,0} \right)  C_z \\ \nonumber
&\quad + \log(\lambda_{N+1}) (\ell_{N+1}^{\alpha}+ \lambda_N^{-\kappa} ) C_z^2, \\ \label{Error:Time}
\mathfrak{E}_{N+1}^{\mathrm{time}} &= \log(\lambda_{N+1}) \lambda_{N+1}^{-3/2} r_N^{1/2} \ell_{N+1}^{-1}, \\ \label{Error:Dissi}
\mathfrak{E}_{N+1}^{\mathrm{dis}} &= \log(\lambda_{N+1}) \abs{\nu} \lambda_{N+1}^{\gamma-3/2}   r_{N}^{1/4}, \\ \label{Error:Trans}
\mathfrak{E}_{N+1}^{\mathrm{trans}}  &= \log(\lambda_{N+1})\lambda_{N}^{1/2} \lambda_{N+1}^{-1/2}  r_N^{1/2} \mathfrak{S}_{N,0}, \\ \label{Error:Sto}
\mathfrak{E}_{N+1}^{\mathrm{sto}} &= \log(\lambda_{N+1}) \lambda_{N+1}^{-1/2}  r_N^{1/2} C_z.
    \end{align}
\end{subequations}
Here $C_z$ is given by~\eqref{def:ConstantPhi} and 
\begin{align}\label{def:SumDec}
\mathfrak{S}_{N,\alpha} := \sum_{n=1}^{N} 4^{N-n} \ell_n^{-\alpha} r_{n-1}^{1/2}.
\end{align}
\end{proposition}

Proposition~\ref{prop:ResidualPropagation} can be used recursively provided that $r_{N+1}$ dominates the right hand side of~\eqref{eq:ResidualPropagation} and the control parameter obey~\eqref{R1},~\eqref{R2} and~\eqref{R3} also at later times. The former is trivially the case if the right hand side defines $r_{N+1}$. However, only if $r_N$ vanishes asymptotically one can hope for a limit $g = \lim_N g_{\leq N}$ that satisfies~\eqref{eq:SolutionConcept}. Finding a suitable family of controls that achieves this is a difficult task. One needs to carefully balance the blow up of $(\lambda_n, \mu_n)$ and the convergence of $(\ell_n, r_n)$. We present a possible choice in Section~\ref{sec:ExampleControls}.

\subsection{Proof of Proposition~\ref{prop:ResidualPropagation}}
The proof proceeds as the one of~\cite[Proposition~5.1]{hofmanovaconvex6}. See also~\cite[Proposition~3.1]{hofmanovaconvex4} for more details. However, contrary to~\cite[Proposition~5.1]{hofmanovaconvex6} we have not yet fixed an explicit choice for the control parameters. 

\subsubsection{Control of $g_{\leq N+1}$ in terms of building blocks}
By definition~\eqref{def:gLeQn} we have full control on the next step $g_{\leq N+1}$ in terms of the previous step and the explicit building block $g_{N+1}$. This allows for the recursive estimate (recall~\eqref{R2})
\begin{align*}
\seminorm{g_{\leq N+1}}_{L^{2m}_\omega C_{\mathrm{loc}}^\alpha B^{1/2}_{\infty,1} } &\leq \seminorm{g_{\ell_{N+1}} }_{L^{2m}_\omega C_{\mathrm{loc}}^\alpha B^{1/2}_{\infty,1} } + \seminorm{ g_{N+1}}_{L^{2m}_\omega C_{\mathrm{loc}}^\alpha B^{1/2}_{\infty,1} }\\
&\leq 4 \seminorm{g_{\leq N} }_{L^{2m}_\omega C_{\mathrm{loc}}^\alpha B^{1/2}_{\infty,1} } + \seminorm{ g_{N+1}}_{L^{2m}_\omega C_{\mathrm{loc}}^\alpha B^{1/2}_{\infty,1} } \\
&\leq \sum_{n=1}^{N+1} 4^{N+1-n} \seminorm{ g_{n}}_{L^{2m}_\omega C_{\mathrm{loc}}^\alpha B^{1/2}_{\infty,1} }.
\end{align*}
Using~\eqref{eq:TimeSpace_Besov}, respectively~\eqref{eq:EstimateBuildingBlock}, and~\eqref{ass:Induction} 
\begin{subequations} \label{eq:SumBoth}
\begin{align} \label{eq:SumAlpha}
\seminorm{g_{\leq N+1}}_{L^{2m}_\omega C_{\mathrm{loc}}^\alpha B^{1/2}_{\infty,1} } &\lesssim \sqrt{(C_0 + 1) }\mathfrak{S}_{N+1,\alpha} , \\ \label{eq:SumZero}
\norm{g_{\leq N+1}}_{L^{2m}_\omega C_{\mathrm{loc}} B^{1/2}_{\infty,1} } &\lesssim \sqrt{(C_0 + 1) }\mathfrak{S}_{N+1,0},
\end{align}
\end{subequations}
where $\mathfrak{S}_{N,\alpha}$ is defined by~\eqref{def:SumDec}.

\subsubsection{Decomposition of new residual}
A mollification of~\eqref{eq:Residual} on the scale~$\ell_{N+1}$ at level $N$ yields
\begin{align} \label{eq:Mollification}
\nabla q_{\ell_{N+1}}  \mathring{\approx}  -\partial_t \mathcal{R} g_{\ell_{N+1}} + \nabla^{\perp} (g_{\ell_{N+1}} + z_{N+1}) \Lambda (g_{\ell_{N+1}} + z_{N+1}) -\nu \Lambda^{\gamma -1} \nabla g_{\ell_{N+1}} -\mathrm{R_{com}},
\end{align}
where
\begin{align*}
\mathrm{R_{com}} &= \nabla^{\perp} (g_{\ell_{N+1}} + z_{N+1}) \Lambda (g_{\ell_{N+1}} + z_{N+1}) \\
&\quad - \nabla^{\perp} (g_{\leq N} + z_{N}) \Lambda (g_{\leq N} + z_{N})\big) *_t \psi_{\ell_{N+1}}.
\end{align*}
Using the explicit update rule~\eqref{def:gLeQn} in~\eqref{eq:Residual}
\begin{align*}
\nabla q_{N+1} &\mathring{\approx} -\partial_t \mathcal{R} (g_{\ell_{N+1}} + g_{N+1} ) + \nabla^{\perp} (g_{\ell_{N+1}} + g_{N+1} + z_{N+1}) \Lambda (g_{\ell_{N+1}} + g_{N+1}  + z_{N+1})\\
&\hspace{4em}- \nu \Lambda^{\gamma -1} \nabla (g_{\ell_{N+1}} + g_{N+1} ) \\
&= \underbrace{-\partial_t \mathcal{R} g_{\ell_{N+1}} + \nabla^{\perp} (g_{\ell_{N+1}} + z_{N+1}) \Lambda (g_{\ell_{N+1}} + z_{N+1})- \nu \Lambda^{\gamma -1} \nabla g_{\ell_{N+1}} }_{\mathring{\approx} \nabla q_{\ell_{N+1}} + \mathrm{R_{com}} \text{ by~\eqref{eq:Mollification}} } + \nabla^{\perp} g_{N+1} \Lambda g_{N+1} \\
&\quad -\partial_t \mathcal{R} g_{N+1}  +  \nabla^{\perp} (g_{\ell_{N+1}} + g_{N+1} + z_{N+1}) \Lambda (g_{\ell_{N+1}} + g_{N+1}  + z_{N+1})- \nu \Lambda^{\gamma -1} \nabla g_{N+1} 
\\
&\quad -  \nabla^{\perp} (g_{\ell_{N+1}} + z_{N+1}) \Lambda (g_{\ell_{N+1}} + z_{N+1}) -   \nabla^{\perp} g_{N+1} \Lambda g_{N+1}  \\
&\mathring{\approx} \underbrace{\nabla q_{\ell_{N+1}} +  \nabla^{\perp} g_{N+1} \Lambda g_{N+1} }_{\mathring{\approx} \nabla q_M} +  \mathrm{R_{com}} -\partial_t \mathcal{R} g_{N+1}  - \nu \Lambda^{\gamma -1} \nabla g_{N+1}  \\
&\quad + \underbrace{\nabla^{\perp} g_{N+1} \Lambda g_{\ell_{N+1}}+  \nabla^{\perp} g_{\ell_{N+1}} \Lambda g_{N+1} }_{\mathring{\approx} \nabla q_T} +  \underbrace{\nabla^{\perp} g_{N+1} \Lambda z_{N+1} +  \nabla^{\perp} z_{N+1}   \Lambda g_{N+1} }_{\mathring{\approx} \nabla q_S}.
\end{align*}
Notice that $\nabla q_{N+1}$ is only defined up to a smooth scalar function $\nabla^\perp p$. Therefore, we subtract the mean-value of each summand on the right hand side without changing the notation. On mean-free functions we can invert the gradient $\Delta^{-1} \nabla \cdot = \nabla^{-1} $.

\subsubsection{Mismatch error}
Due to~\eqref{R3} and~\eqref{R1}, we can follow the line of argumentation used in~\cite[Proposition~3.1]{MR4340931} to derive
\begin{align*}
\norm{q_M}_{X} \lesssim \left\{ \log \mu_{N+1} \left( \frac{\lambda_N}{\mu_{N+1}}  + \frac{\mu_{N+1}}{\lambda_{N+1}} \right)^2 + \frac{\lambda_N}{\lambda_{N+1}} \right\}(C_0 + 1) \chi_{\ell_{N+1}} .
\end{align*}
This together with~\eqref{eq:EstimateChi} gives

\begin{align*}
\norm{q_M}_{L^m C_{\mathrm{loc}} X} &\lesssim \left\{ \log \mu_{N+1} \left( \frac{\lambda_N}{\mu_{N+1}}  + \frac{\mu_{N+1}}{\lambda_{N+1}}\right)^2 + \frac{\lambda_N}{\lambda_{N+1}} \right\}(C_0 + 1) r_{N}.
\end{align*}

\subsubsection{Commutator error} We decompose the commutator error
\begin{align*}
\mathrm{R_{com}} &= \left( \nabla^{\perp} (z_{N}-z_{N+1}) \Lambda (g_{\leq N} + z_{N+1})+ \nabla^\perp(g_{\leq N} + z_{N}) \Lambda (z_{N}-z_{N+1}) \right) *_t \psi_{\ell_{N+1}} \\
&+ \nabla^{\perp} (g_{\ell_{N+1}}- g_{\leq N}) \Lambda (g_{\leq N} + z_{N+1}) + \nabla^{\perp} (g_{\ell_{N+1}} + z_{N+1}) \Lambda (g_{\ell_{N+1}} -g_{\leq N}) \\
&+ \nabla^{\perp} (g_{\leq N} + z_{N+1}) \Lambda (g_{\leq N} + z_{N+1})  \\
&\hspace{3em} - \big( \nabla^{\perp} (g_{\leq N} + z_{N+1}) \Lambda (g_{\leq N} + z_{N+1})\big) *_t \psi_{\ell_{N+1}} \\
&=: \mathrm{R}_{\mathrm{com}}^1 + \mathrm{R}_{\mathrm{com}}^2 + \mathrm{R}_{\mathrm{com}}^3.
\end{align*}
Notice that $ \mathrm{R}_{\mathrm{com}}^i$, $i \in \{1,2,3\}$, is frequency localized to~$\{k: \abs{k} \leq 2\lambda_{N+1}\}$. Thus, by Lemma~\ref{lem:LinfQuantified}, Poincaré's and H\"older's inequalities
\begin{align*}
&\norm{\Delta^{-1} \nabla \cdot \mathrm{R}_{\mathrm{com}}^1}_{L^m_\omega C_{\mathrm{loc}} X} \lesssim \log(\lambda_{N+1}) \norm{ \nabla^{\perp} (z_{N}-z_{N+1}) \Lambda( g_{\leq N} + z_{N+1}) }_{L^m_\omega C_{\mathrm{loc}} L^\infty} \\
&\hspace{10em} + \log(\lambda_{N+1})   \norm{\nabla^\perp( g_{\leq N}+ z_{N}) \Lambda (z_{N}-z_{N+1}) ]  }_{L^m_\omega C_{\mathrm{loc}} L^\infty} \\
&\hspace{3em}\lesssim  \log(\lambda_{N+1}) \left( \lambda_N^{1/2} \norm{g_{\leq N}}_{L^{2m}_\omega C_{\mathrm{loc}} B^{1/2}_{\infty,1}} + \norm{z}_{L^{2m}_\omega C_{\mathrm{loc}} C^{1+}_x} \right) \lambda_N^{-\kappa} \norm{z}_{L^{2m}_\omega C_{\mathrm{loc}} C^{1+\kappa}_x}.
\end{align*}
Let us stress at this point that it is here, where we require the condition $\kappa>1/2$ to hold in order to compensate for the diverging factor $\lambda_N^{1/2}$ above. This explains the additional regularity requirement expressed in $\kappa$ compared with \cite{hofmanovaconvex6}, where a more detailed analysis of the noise term is performed (however, this more detailed analysis exploits the Gaussianity of the underlying noise $W$ and thus does not extend to the setting of more general noises we are concerned with here).
Similarly,
\begin{align*}
\norm{\Delta^{-1} \nabla \cdot \mathrm{R}_{\mathrm{com}}^2}_{L^m_\omega C_{\mathrm{loc}} X}&\lesssim \log(\lambda_{N+1}) \lambda_N^{1/2} \ell_{N+1}^{\alpha} \seminorm{g_{\leq N}}_{L^{2m}_\omega C_{\mathrm{loc}}^\alpha B^{1/2}_{\infty,1}} \\
&\quad \times \left( \lambda_N^{1/2} \norm{g_{\leq N}}_{L^{2m}_\omega C_{\mathrm{loc}} B^{1/2}_{\infty,1}} + \norm{z}_{L^{2m}_\omega C_{\mathrm{loc}} C^{1+}_x} \right)
\end{align*}
and
\begin{align*}
&\norm{\Delta^{-1} \nabla \cdot \mathrm{R}_{\mathrm{com}}^3}_{L^m_\omega C_{\mathrm{loc}} X} \\
&\hspace{2em} \lesssim \log(\lambda_{N+1}) \ell_{N+1}^{\alpha} \seminorm{\nabla^{\perp} (g_{\leq N} + z_{N+1}) \Lambda (g_{\leq N} + z_{N+1}) }_{L^m_\omega C_{\mathrm{loc}}^\alpha L^\infty} \\
&\hspace{2em} \leq \log(\lambda_{N+1}) \ell_{N+1}^{\alpha} 
\norm{\nabla^{\perp} (g_{\leq N} + z_{N+1})}_{L^{2m}_\omega C_{\mathrm{loc}} L^\infty}  \seminorm{\Lambda (g_{\leq N} + z_{N+1}) }_{L^{2m}_\omega C_{\mathrm{loc}}^\alpha L^\infty} \\
&\hspace{2em}  \quad + \log(\lambda_{N+1}) \ell_{N+1}^{\alpha} 
\norm{\Lambda (g_{\leq N} + z_{N+1})}_{L^{2m}_\omega C_{\mathrm{loc}} L^\infty}  \seminorm{ \nabla^{\perp} (g_{\leq N} + z_{N+1})}_{L^{2m}_\omega C_{\mathrm{loc}}^\alpha L^\infty} \\
&\hspace{2em}  \leq \log(\lambda_{N+1}) \ell_{N+1}^{\alpha} \left( \lambda_N^{1/2}
\norm{ g_{\leq N}}_{L^{2m}_\omega C_{\mathrm{loc}} B^{1/2}_{\infty,1}} + \norm{z}_{L^{2m}_\omega C_{\mathrm{loc}} C^{1+}_x} \right)  \\
&\hspace{5em} \times \left( \lambda_N^{1/2} \seminorm{g_{\leq N}}_{L^{2m}_\omega C_{\mathrm{loc}}^\alpha B^{1/2}_{\infty,1}} + \seminorm{z }_{L^{2m}_\omega C_{\mathrm{loc}}^\alpha C^{1+}_x} \right).
\end{align*}

All together, using~\eqref{eq:SumBoth},
\begin{align*}
\norm{\Delta^{-1} \nabla \cdot \mathrm{R}_{\mathrm{com}}}_{L^m_\omega C_{\mathrm{loc}} X} &\lesssim \log(\lambda_{N+1}) \lambda_N \ell_{N+1}^{\alpha}(C_0 + 1)\mathfrak{S}_{N,\alpha} \mathfrak{S}_{N,0} \\
&\quad+  \log(\lambda_{N+1}) \lambda_N^{1/2} \sqrt{(C_0 + 1) }\left( \ell_{N+1}^{\alpha}\mathfrak{S}_{N,\alpha} + \lambda_N^{-\kappa} \mathfrak{S}_{N,0} \right)  C_z \\
&\quad + \log(\lambda_{N+1}) (\ell_{N+1}^{\alpha}+ \lambda_N^{-\kappa} ) C_z^2.
\end{align*}

\subsubsection{Time error}
Since $g_{N+1}$ is frequency localized to an annulus of size~$\lambda_{N+1}$
\begin{align*}
\norm{\Delta^{-1} \nabla \cdot \partial_t \mathcal{R} g_{N+1}}_{L^m_\omega C_{\mathrm{loc}} X} &\lesssim  \log(\lambda_{N+1}) \lambda_{N+1}^{-1} \norm{\partial_t g_{N+1}}_{L^m_\omega C_{\mathrm{loc}} L^\infty}.
\end{align*}
Using~\eqref{eq:TimeDerivative} and~\eqref{ass:Induction}
\begin{align*}
\norm{\Delta^{-1} \nabla \cdot \partial_t \mathcal{R} g_{N+1}}_{L^m_\omega C_{\mathrm{loc}} X} \lesssim  \log(\lambda_{N+1}) \lambda_{N+1}^{-3/2} r_N^{1/2} \ell_{N+1}^{-1}\left( \sqrt{C_0 + 1} + \frac{1}{\sqrt{C_0 - 1}} \right).
\end{align*}

\subsubsection{Dissipation error}
As $g_{N+1}$ is frequency localized to an annulus of size~$\lambda_{N+1}$
\begin{align*}
\norm{\Delta^{-1} \nabla \cdot \nu \Lambda^{\gamma -1} \nabla g_{N+1} }_{L^m C_{\mathrm{loc}} X} \lesssim \log(\lambda_{N+1}) \abs{\nu} \lambda_{N+1}^{\gamma-3/2}  \norm{g_{N+1}}_{L^m_\omega C_{\mathrm{loc}} B^{1/2}_{\infty,1}}.
\end{align*}
Using H\"older's inequality,~\eqref{eq:EstimateBuildingBlock} and~\eqref{ass:Induction}
\begin{align*}
\norm{\Delta^{-1} \nabla \cdot \nu \Lambda^{\gamma -1} \nabla g_{N+1} }_{L^m C_{\mathrm{loc}} X} \lesssim \log(\lambda_{N+1}) \abs{\nu}\lambda_{N+1}^{\gamma-3/2} \left( (C_0 + 1) r_{N} \right)^{1/4}.
\end{align*}

\subsubsection{Transport error}
Notice that $\nabla^{\perp} g_{N+1} \Lambda g_{\ell_{N+1}}$ and $ \nabla^{\perp} g_{\ell_{N+1}} \Lambda g_{N+1}$ are frequency localized to an annulus of size~$\lambda_{N+1}$. H\"older's inequality,~\eqref{eq:EstimateBuildingBlock} and~\eqref{eq:SumZero} imply

\begin{align*}
&\norm{\Delta^{-1} \nabla \cdot[\nabla^{\perp} g_{N+1} \Lambda g_{\ell_{N+1}}+  \nabla^{\perp} g_{\ell_{N+1}} \Lambda g_{N+1}]}_{L^m_\omega C_{\mathrm{loc}} X} \\
&\quad \lesssim \log(\lambda_{N+1}) \lambda_{N+1}^{-1} \norm{\nabla^{\perp} g_{N+1} \Lambda g_{\ell_{N+1}}+  \nabla^{\perp} g_{\ell_{N+1}} \Lambda g_{N+1}}_{L^m_\omega C_{\mathrm{loc}} L^\infty} \\
&\quad \leq \log(\lambda_{N+1})\lambda_{N}^{1/2} \lambda_{N+1}^{-1/2} \norm{g_{N+1}}_{L^{2m}_\omega C_{\mathrm{loc}} B^{1/2}_{\infty,1}} \norm{g_{\ell_{N+1}}}_{L^{2m}_\omega C_{\mathrm{loc}} B^{1/2}_{\infty,1}} \\
&\quad \lesssim \log(\lambda_{N+1})\sqrt{C_0 + 1}  \lambda_{N}^{1/2} \lambda_{N+1}^{-1/2}  r_N^{1/2} \norm{g_{\leq N}}_{L^{2m}_\omega C_{\mathrm{loc}} B^{1/2}_{\infty,1}} \\
&\quad \lesssim \log(\lambda_{N+1})(C_0 + 1)  \lambda_{N}^{1/2} \lambda_{N+1}^{-1/2}  r_N^{1/2} \mathfrak{S}_{N,0}.
\end{align*}

\subsubsection{Stochastic error} 
Since $\nabla^{\perp} g_{N+1} \Lambda z_{N+1}$ and $\nabla^{\perp} z_{N+1}   \Lambda g_{N+1}$ are frequency localized to an annulus of size~$\lambda_{N+1}$, H\"older's inequality and~\eqref{eq:EstimateBuildingBlock}
\begin{align*}
&\norm{\Delta^{-1} \nabla \cdot [\nabla^{\perp} g_{N+1} \Lambda z_{N+1} +  \nabla^{\perp} z_{N+1}   \Lambda g_{N+1}]}_{L^m_\omega C_{\mathrm{loc}} X} \\
&\quad \lesssim \log(\lambda_{N+1}) \lambda_{N+1}^{-1} \norm{\nabla^{\perp} g_{N+1} \Lambda z_{N+1} +  \nabla^{\perp} z_{N+1}   \Lambda g_{N+1} }_{L^{m}_\omega C_{\mathrm{loc}} L^\infty} \\
&\quad \lesssim \log(\lambda_{N+1}) \lambda_{N+1}^{-1/2}  \norm{g_{N+1}}_{L^{2m}_\omega C_{\mathrm{loc}} B^{1/2}_{\infty,1}} \norm{z}_{L^{2m}_\omega C_{\mathrm{loc}} C^{1+}_x} \\
&\quad \lesssim \log(\lambda_{N+1}) \sqrt{C_0 + 1}  \lambda_{N+1}^{-1/2}  r_N^{1/2} C_z.
\end{align*}

Collecting all estimates and choosing $C_\mathrm{it}$ as the maximum of all implicit constants and the terms depending on $C_0$ verifies the claim. From now on, we will no longer indicate the dependence of $C_\mathrm{it}$ on $C_0$.

\subsection{An explicit choice of controls} \label{sec:ExampleControls}
Let $a, b \in \mathbb{N}$, $\beta \in (0,1)$ and define 
\begin{subequations}\label{def:Parameter}
\begin{alignat}{3} 
  &\forall n \geq 0: \hspace{2em} \lambda_n  &&=  a^{(b^n)} ,\hspace{2em} &&r_n = C_{\mathrm{start}}(C_z^2 + 1) \lambda_0^{\beta} \lambda_n^{-\beta}, \\
  &\forall n \geq 1: \hspace{2em} \ell_n &&= \lambda_n^{-1}, \hspace{2em}  &&\mu_n = \sqrt{\lambda_{n-1}\lambda_n},
\end{alignat}
\end{subequations}
where $C_z$ is given by~\eqref{def:ConstantPhi} and $C_{\mathrm{start}} \geq 1$ is a free parameter. It can be used to ensure that the initial residual is dominated by $r_0$.

The next lemma allows to iterate Proposition~\ref{prop:ResidualPropagation} for specific choices of $(a,b,\beta)$.
\begin{lemma} \label{lem:ControlChoice}
Let $m \geq 1$, $\alpha > 0$, $\kappa > 1/2$, $\gamma < 3/2$ and $z \in L^{2m}_\omega\big( C_{\mathrm{loc}} C^{1+\kappa}_x \cap C_{\mathrm{loc}}^\alpha C^{1+}_x \big)$. 

Then for all $C >1$ there exist $a,b \in \mathbb{N}$ and $\beta \in (0,1)$ such that the control choice~\eqref{def:Parameter} satisfies for all $n \geq 1$ the conditions~\eqref{R1},~\eqref{R2},~\eqref{R3} and
\begin{align} \label{eq:ChoiceEstimate}
    C \big( \mathfrak{E}_{n}^{\mathrm{miss}} + \mathfrak{E}_{n}^{\mathrm{com}} + \mathfrak{E}_{n}^{\mathrm{time}} + \mathfrak{E}_{n}^{\mathrm{dis}} + \mathfrak{E}_{n}^{\mathrm{trans}} + \mathfrak{E}_{n}^{\mathrm{sto}} \big) \leq r_{n},
\end{align}
where~$\mathfrak{E}_{n}^{\bullet}$ is defined by~\eqref{Error:Global}.
\end{lemma}

\begin{proof}
We aim to trace the restrictions on $a,b$ and $\beta$ such that~\eqref{R1},~\eqref{R2},~\eqref{R3} and~\eqref{eq:ChoiceEstimate} are realized individually. 

By assumption $C_z$ is finite. Notice that $\lambda_n$ and $\mu_n$ are increasing and, thus, $r_n$ and $\ell_n$ are decreasing. Moreover~\eqref{R1} and~\eqref{R2} are satisfied for all $a,b > 1$. We find that~\eqref{R3} imposes the restrictions
\begin{subequations} \label{ConditionA}
\begin{alignat}{3}
12 \lambda_n  &\leq 4 \mu_{n+1} \hspace{1em} &&\Leftrightarrow  \hspace{1em} 9 &&\leq a^b , \\
10 &\leq \mu_n \hspace{1em} &&\Leftrightarrow \hspace{1em} 100 &&\leq a^{1+b} , \\
48 \lambda_n &\leq \lambda_{n+1} \hspace{1em} &&\Leftrightarrow  \hspace{1em} 48 &&\leq  a^b .
\end{alignat}
\end{subequations}

Observe that $r_{m} = r_n \lambda_m^{-\beta} \lambda_n^{\beta} $ for each $m,n \in \mathbb{N}$. Next, we will find sufficient conditions such that each term on the left hand side of~\eqref{eq:ChoiceEstimate} is bounded by $r_n$. 

\textit{ad~\eqref{Error:Miss}.} As $(1-b)(1-\beta) < 0$ the factor in front of $r_n$,
\begin{align*}
   C \mathfrak{E}_{n}^{\mathrm{miss}}  &= C\left( b^{n-1} (1+b) \log(a) +1 \right) a^{b^{n-1}(1-b)(1-\beta)} r_n,
\end{align*}
can be made arbitrarily small.

Before we proceed let us estimate the sums~\eqref{def:SumDec}. Neglecting the decaying factor and estimating the geometric series show
\begin{align} \label{est:Sums}
    \mathfrak{S}_{n-1,\alpha} &= r_{n}^{1/2} a^{b^{n} \beta/2} \sum_{k=1}^{n-1} 4^{n-1-k} a^{b^{k-1}(\alpha b - \beta/2)} \leq r_{n}^{1/2} a^{b^{n-1}( b\beta/2 + \alpha) } 4^{n}.
\end{align}

\textit{ad~\eqref{Error:Com}.} Using~\eqref{est:Sums},
\begin{align*}
     C\mathfrak{E}_{n}^{\mathrm{com}} &\leq Cb^n \log(a) 4^{2n}  a^{b^{n-1}(1-\alpha (b-1) + \beta b)} r_n \\
     &+ Cb^n \log(a) 4^n \left(a^{b^{n-1}(1/2 - \alpha b + \beta b)}   + a^{b^{n-1}(1/2 - \kappa + \beta b)}  \right) r_n \\
     &+  Cb^n \log(a)\left( a^{b^{n}( \beta- \alpha ) } + a^{b^{n-1}( b\beta- \kappa ) } \right) r_n.
\end{align*}
We need to restrict to
\begin{alignat}{3} \label{ConditionEta1}
1-\alpha (b-1) + \beta b &< 0 \hspace{1em} &&\Leftrightarrow  \hspace{1em} \beta  &&<\frac{ \alpha (b-1) - 1}{b} , \\ \label{ConditionB}
\frac{ \alpha (b-1) - 1}{b} &> 0 \hspace{1em} &&\Leftrightarrow  \hspace{1em} b &&> 1 + \frac{1}{\alpha}, \\ \label{ConditionEta2}
1/2 - \kappa + \beta b &< 0 \hspace{1em} &&\Leftrightarrow \hspace{1em} \beta &&< \frac{1}{b} \left( \kappa - \frac{1}{2} \right).
\end{alignat}

\textit{ad~\eqref{Error:Time}.} Notice that $r_{n}^{1/2} \leq a^{b^n \beta/2} r_n$. Therefore,
\begin{align*}
    C\mathfrak{E}_{n}^{\mathrm{time}} &= Cb^{n} \log(a) a^{b^{n-1}( - b/2  + b\beta/2 - \beta /2 )}  r_{n}^{1/2} \\
    &\leq  Cb^{n} \log(a) a^{b^{n-1}( - b/2  + b\beta - \beta /2 )}  r_{n}.
\end{align*}
We need to impose
\begin{alignat}{3} \label{ConditionEta3}
- b/2  + b\beta - \beta /2  &< 0 \hspace{1em} &&\Leftrightarrow  \hspace{1em} \beta  &&< \frac{b}{2b - 1}.
\end{alignat}

\textit{ad~\eqref{Error:Dissi}.} Notice that $r_{n-1}^{1/4} \leq a^{b^{n-1} ( b \beta - \beta/4)} r_n$. This implies 
\begin{align*}
   C \mathfrak{E}_{n}^{\mathrm{dis}} &\leq C\abs{\nu} b^n \log(a) a^{b^{n-1}(b( \gamma - 3/2 + \beta) - \beta/4)}    r_{n}.
\end{align*}
The condition reads
\begin{alignat}{3} \label{ConditionEta4}
b( \gamma - 3/2 + \beta) - \beta/4  &< 0 \hspace{1em} &&\Leftrightarrow  \hspace{1em} \beta  &&< \frac{4b}{4b-1} \left( \frac{3}{2} - \gamma \right).
\end{alignat}

\textit{ad~\eqref{Error:Trans}.} Using~\eqref{est:Sums}
\begin{align*}
   C \mathfrak{E}_{n}^{\mathrm{trans}}     &\leq C b^n \log(a) 4^n a^{b^{n-1}( 1/2 - b/2 + b \beta - \beta /2)}  r_{n}.
\end{align*}

We impose
\begin{alignat}{3} \label{ConditionEta5}
1/2 - b/2 + b \beta - \beta /2  &< 0 \hspace{1em} &&\Leftrightarrow  \hspace{1em} \beta  &&< \frac{b-1}{2b -1}.
\end{alignat}

\textit{ad~\eqref{Error:Sto}.} Notice that $r_{n-1}^{1/2} C_z \leq a^{b^{n-1}( b \beta - \beta/2)} r_n$. Therefore,
\begin{align*}
   C  \mathfrak{E}_{n}^{\mathrm{sto}} &\leq  Cb^n \log(a) a^{b^{n-1}( -b/2 + b\beta - \beta/2 )}  r_{n}.
\end{align*}
This calls for the condition
\begin{alignat}{3} \label{ConditionEta6}
-b/2 + b\beta - \beta/2   &< 0 \hspace{1em} &&\Leftrightarrow  \hspace{1em} \beta  &&  < \frac{b}{2b-1}.
\end{alignat}

Lastly, we need to verify that all conditions can be satisfied simultaneously. First the value of $b$ is determined by~\eqref{ConditionB}. This allows to find $\beta > 0$ such that
~\eqref{ConditionEta1}, ~\eqref{ConditionEta2} - \eqref{ConditionEta6} hold. Now, $b$ and $\beta$ are fixed. Let $\varsigma > 0$ quantify the slowest convergence, i.e.,
\begin{align*}
    -\varsigma &= \max \big\{ (1-b)(1-\beta),\, 1-\alpha (b-1) + \beta b,\, 1/2 - \kappa + \beta b, \\
    &\hspace{4em} - b/2  + b\beta - \beta /2, \,b( \gamma - 3/2 + \beta) - \beta/4,\, 1/2 - b/2 + b \beta - \beta /2  \big\}.
\end{align*}

As $\mathbb{N} \ni n \mapsto (4b)^{2n} a^{-b^{n-1}\varsigma}$ is decreasing for $a \geq a_0(b,\varsigma)$, we estimate
\begin{align*}
     &C \big( \mathfrak{E}_{n}^{\mathrm{miss}} + \mathfrak{E}_{n}^{\mathrm{com}} + \mathfrak{E}_{n}^{\mathrm{time}} + \mathfrak{E}_{n}^{\mathrm{dis}} + \mathfrak{E}_{n}^{\mathrm{trans}} + \mathfrak{E}_{n}^{\mathrm{sto}} \big) \\
     &\hspace{3em} \leq C \log(a) \big(10 + \abs{\nu} \big) (4b)^{2n} a^{-b^{n-1}\varsigma} r_n \\
     &\hspace{3em} \leq C \log(a) \big(10 + \abs{\nu} \big) (4b)^{2} a^{-\varsigma} r_n.
\end{align*}
The value of $a$ is chosen such that~\eqref{ConditionA}, $ a\geq a_0(b,\varsigma)$ and
\begin{align}
 C \log(a) \big(10 + \abs{\nu} \big) (4b)^{2} a^{-\varsigma} \leq 1
\end{align}
hold. 
\end{proof}

\begin{remark}
    We want to stress that Lemma~\ref{lem:ControlChoice} does not depend on the value~$C_{\mathrm{start}}$ nor the explicit value $C_z$. In particular, it is independent of $m \geq 1$ as long as $C_z < \infty$. 

    Notice that the exponent $b$ is purely determined by the $\alpha$-H\"older regularity in time of $z$, cf.~\eqref{ConditionB}. In other words, a time irregular stochastic term ($\alpha$ small) forces us to introduce higher frequencies at a faster rate, i.e., $n \mapsto \lambda_n \sim a^{\alpha^{-n}}$ increases rapidly.
\end{remark}

We are ready to proof convergence of the iteration~\eqref{def:Iteration}.

\begin{theorem}
\label{convergenceTheorem}
Let $m \geq 1$, $\alpha > 0$, $\kappa > 1/2$, $\gamma < 3/2$ and $z \in L^{2m}_\omega\big( C_{\mathrm{loc}}C^{1+\kappa}_x \cap C_{\mathrm{loc}}^\alpha C^{1+}_x \big)$. 

Then there exist $a,b \in \mathbb{N}$, $\beta \in(0,1)$ and $C_{\mathrm{start}} \geq 1$ such that for all $n' \geq n \geq 0$
\begin{align} \label{eq:EstimateQ}
    \norm{q_{n}}_{L^m_\omega C_{\mathrm{loc}}X} &\leq r_n, \\ \label{eq:EstimateG}
    \norm{g_{\leq n'} - g_{\leq n}}_{L^{2m}_\omega C_{\mathrm{loc}} B^{1/2}_{\infty,1} } &\lesssim \sqrt{r_n}. 
\end{align}
In particular, $g = \lim_{n \to \infty} g_{\leq n} \in L^{2m}_\omega C_{\mathrm{loc}} B^{1/2}_{\infty,1}$ exists, solves~\eqref{intro:Main} and satisfies 
\begin{align} \label{eq:NormEstimateLimit}
    \norm{g}_{L^{2m}_\omega C_{\mathrm{loc}} B^{1/2}_{\infty,1} } \lesssim \sqrt{C_{\mathrm{start}}} \left( \norm{z}_{L^{2m}_\omega C_{\mathrm{loc}} C^{1+\kappa}_x} + \seminorm{z }_{L^{2m}_\omega C^\alpha_{\mathrm{loc}} C^{1+}_x} + 1 \right). 
\end{align}
\end{theorem}  

\begin{proof}
We apply Lemma~\ref{lem:ControlChoice} for $C = C_{\mathrm{it}}$ from Proposition~\ref{prop:ResidualPropagation} and obtain a parameter configuration $(a,b,\beta)$ such that~\eqref{R1},~\eqref{R2},~\eqref{R3} and~\eqref{eq:ChoiceEstimate} hold for all $n \geq 1$. We can choose $(a,b,\beta)$ such that ~\eqref{ConditionEta1} holds. 

Now, we proceed via induction. First, we argue that it is possible to start the iteration. By~\eqref{def:InitResidual}, the frequency locality of $z_0$ and Lemma~\ref{lem:LinfQuantified}, Poincaré's and H\"older's inequalities
    \begin{align*}
        \norm{q_0}_{L^m_\omega C_{\mathrm{loc}}X} &= \norm{ \Delta^{-1} \nabla \cdot [\nabla^\perp z_0 \Lambda z_0] }_{L^m_\omega C_{\mathrm{loc}}X} \\
        &\lesssim \log(\lambda_0) \norm{ \Delta^{-1} \nabla \cdot [\nabla^\perp z_0 \Lambda z_0] }_{L^m_\omega C_{\mathrm{loc}}L^\infty_x} \\
       &\lesssim \log(\lambda_0) \norm{\nabla^\perp z_0 \Lambda z_0] }_{L^m_\omega C_{\mathrm{loc}}L^\infty_x}  \\
       &\lesssim \log(\lambda_0) \norm{ z_0}_{L^{2m}_\omega C_{\mathrm{loc}}C^{1+}_x}^2 \\
       &\lesssim \log(\lambda_0) (C_z^2 + 1).
    \end{align*}
Therefore, there exists $C_\mathrm{start}$ such that $\norm{q_0}_{L^m_\omega C_{\mathrm{loc}}X} \leq r_0$.

Since Lemma~\ref{lem:ControlChoice} guarantees that the assumptions of Proposition~\ref{prop:ResidualPropagation} are satisfied, the induction step follows and we have shown~\eqref{eq:EstimateQ}. 

Next, we verify~\eqref{eq:EstimateG}. Notice that, due to~\eqref{def:gLeQn} for arbitrary $k \geq 0$,
\begin{align*}
    \norm{g_{\leq k+1} - g_{\leq k}}_{L^{2m}_\omega C_{\mathrm{loc}} B^{1/2}_{\infty,1} } \lesssim \ell_{k+1}^\alpha \seminorm{g_{\leq k}}_{L^{2m}_\omega C^\alpha_{\mathrm{loc}} B^{1/2}_{\infty,1} } +  \norm{g_{k+1}}_{L^{2m}_\omega C_{\mathrm{loc}} B^{1/2}_{\infty,1} }.
\end{align*}
Therefore, using~\eqref{eq:EstimateBuildingBlock} and~\eqref{eq:SumAlpha},
\begin{align*}
    \norm{g_{\leq k+1} - g_{\leq k}}_{L^{2m}_\omega C_{\mathrm{loc}} B^{1/2}_{\infty,1} } \lesssim \ell_{k+1}^\alpha  \mathfrak{S}_{k,\alpha} + \sqrt{r_{k}}.
\end{align*}
Recall~\eqref{est:Sums} and~\eqref{ConditionEta1}. This allows to bound
\begin{align*}
    \ell_{k+1}^\alpha  \mathfrak{S}_{k,\alpha} \leq \sqrt{r_{k+1}} a^{b^{k}( b\beta/2 - \alpha(b-1) ) } 4^{k+1} \leq \sqrt{r_{k+1}} \leq \sqrt{r_{k}}.
\end{align*}
Notice that $r_{m} = r_n \lambda_m^{-\beta} \lambda_n^{\beta} $ for all $m,n \in \mathbb{N}$. Therefore,
\begin{align*}
     \norm{g_{\leq n'} - g_{\leq n}}_{L^{2m}_\omega C_{\mathrm{loc}} B^{1/2}_{\infty,1} } &\leq  \sum_{k=n}^{n'-1}  \norm{g_{\leq k+1} - g_{\leq k}}_{L^{2m}_\omega C_{\mathrm{loc}} B^{1/2}_{\infty,1} } \\
     &\lesssim \sqrt{r_n}  \sum_{k=n}^{n'-1} a^{-b^n \beta (b^{k-n}-1)} \lesssim \sqrt{r_n},
\end{align*}
where we used the estimate
\begin{align*}
     \sum_{k=n}^{n'-1} a^{ -\beta b^n(b^{k-n}-1)} \leq  \sum_{k=0}^{n'-n-1} a^{-\beta (b-1)k} \leq \frac{1}{1-a^{-\beta (b-1)}}.
\end{align*}

Lastly, we need to check that $g$ solves~\eqref{intro:Main} and satisfies~\eqref{eq:NormEstimateLimit}. First notice that, since $z_n = P_{\lambda_n} z$ and $C^{1+\kappa} \hookrightarrow \dot{H}^{3/2}$,
\begin{align*}
    \norm{z_n - z}_{\dot{H}^{1/2}}^2 \leq \sum_{ \abs{k} \geq \lambda_n}  \abs{k} \abs{\widehat{z}(k)}^2 \leq \lambda_n^{-2} \norm{z}_{\dot{H}^{3/2}}^2 \lesssim \lambda_n^{-2} \norm{z}_{C^{1+\kappa}_x}^2.
\end{align*}
This implies $z_n \rightarrow z \in L^{2m}_\omega C_{\mathrm{loc}} \dot{H}^{1/2}$. Moreover, as $B^{1/2}_{\infty,1} \hookrightarrow \dot{H}^{1/2}$ the strong convergence of $g_{\leq n} \in L^{2m}_\omega C_{\mathrm{loc}} B^{1/2}_{\infty,1} $ implies $g_{\leq n} + z_n \rightarrow g + z \in L^{2m}_\omega C_{\mathrm{loc}} \dot{H}^{1/2}$.

Let $\xi \in C^\infty_c(\mathbb{R}\times \mathbb{T}^2)$. By~\eqref{eq:Residual} it holds
\begin{align*}
 \langle -\nabla q_{n} -\partial_t \mathcal{R} g_{\leq n} + \nabla^{\perp} (g_{\leq n} + z_n) \Lambda (g_{\leq n} + z_n)- \nu \Lambda^{\gamma -1} \nabla g_{\leq n}, \nabla \xi \rangle  = 0.
\end{align*}
Multiply by $\xi_\omega \in L^\infty_\omega$, take expectation and split up the sum. The linear terms converge due to weak convergence
\begin{align*}
   \mathbb{E}\left[ \langle -\nabla q_{n} -\partial_t \mathcal{R} g_{\leq n}- \nu \Lambda^{\gamma -1} \nabla g_{\leq n}, \nabla \xi \rangle \xi_\omega \right]  \rightarrow   \mathbb{E}\left[ \langle -\partial_t \mathcal{R} g- \nu \Lambda^{\gamma -1} \nabla g, \nabla \xi \rangle \xi_\omega \right].
\end{align*}
Using Lemma~\ref{lem:Commutator} we find for the non-linear term
\begin{align*}
     &\mathbb{E}\left[ \langle  \nabla^{\perp} (g_{\leq n} + z_n) \Lambda (g_{\leq n} + z_n) - \nabla^{\perp} (g + z) \Lambda (g + z), \nabla \xi \rangle \xi_\omega \right]  \\
     &\hspace{2em} \leq \mathbb{E}\left[   \norm{(g_{\leq n} + z_n) -(g + z)}_{C_{\mathrm{loc}} \dot{H}^{1/2}}^2 \norm{ \xi}_{L^1_t \dot{H}^4}  \abs{\xi_\omega} \right] \\
     &\hspace{2em}\leq \norm{(g_{\leq n} + z_n) -(g + z)}_{L^{2m}_\omega C_{\mathrm{loc}} \dot{H}^{1/2}}^2 \norm{ \xi}_{L^1_t \dot{H}^4}  \norm{\xi_\omega}_{L^{\infty}_\omega}.
\end{align*}
Thus, as $\xi_\omega$ was arbitrary, we have verified that $g$ satisfies~\eqref{eq:SolutionConcept}. 

The estimate~\eqref{eq:NormEstimateLimit} follows from~\eqref{eq:EstimateG} by noticing that $g_{\leq 0} = 0$.

\end{proof}

In Theorem~\ref{convergenceTheorem} we traced the concentration speed of $g_{\leq n}$ in $L^{2m}_\omega C_{\mathrm{loc}} B^{1/2}_{\infty,1}$ as it is sufficient for the identification of the non-linearity. However, the explicit structure of the solution allows to read of improved regularity based on the parameter choice. 

\begin{corollary}[Regularity] \label{cor:Reg}
Let the framework of Theorem~\ref{convergenceTheorem} be satisfied. Moreover, let $(a,b,\beta,C_{\mathrm{start}})$ be given by Theorem~\ref{convergenceTheorem} and $g$ denote the corresponding solution. 
Let $\vartheta_t, \vartheta_x \in [0,1)$ satisfy $\vartheta_t + \vartheta_x < 1/2 + \beta/(2b)$ and either
\begin{enumerate}
    \item (high spatial regularity) 
    \begin{align} \label{cond:HighSpace}
   \vartheta_x \geq \frac{1}{2} - \alpha +  \frac{\beta}{2b} \quad \text{and} \quad  \vartheta_t     < \alpha - \frac{1}{b}\left( \vartheta_x - \left( \frac{1}{2} - \alpha +\frac{\beta}{2b}\right)  \right),
\end{align}
\item (high temporal regularity) or 
    \begin{align} \label{cond:HighTime}
   \vartheta_x < \frac{1}{2} - \alpha +  \frac{\beta}{2b}\quad \text{and} \quad  \vartheta_t     < \alpha.
\end{align}
\end{enumerate}
Then $g \in L^{2m}_\omega C^{\vartheta_t}_{\mathrm{loc}} C^{\vartheta_x}$ with
    \begin{align} \label{eq:RegularityCor}
        \seminorm{g}_{L^{2m}_\omega C_{\mathrm{loc}}^{\vartheta_t} C^{\vartheta_x}} \lesssim \sqrt{C_{\mathrm{start}}} \left( \norm{z}_{L^{2m}_\omega C_{\mathrm{loc}} C^{1+\kappa}_x} + \seminorm{z }_{L^{2m}_\omega C^\alpha_{\mathrm{loc}} C^{1+}_x} + 1 \right).
    \end{align}
\end{corollary}
\begin{proof}
Let $(a,b,\beta, C_{\mathrm{start}})$ be a parameter configuration given by Theorem~\ref{convergenceTheorem} and $g$ its corresponding convex integration solution. In particular, it holds $ \norm{q_{n}}_{L^m_\omega C_{\mathrm{loc}}X} \leq r_n$ for all $n$. Then, using~\eqref{def:gLeQn} and~\eqref{eq:TimeSpace_Hoelder}
\begin{align*}
    \seminorm{g_{\leq N+1} - g_{\leq N} }_{L^{2m}_\omega C_{\mathrm{loc}}^{\vartheta_t} C^{\vartheta_x}} &\lesssim \ell_{N+1}^{\alpha - \vartheta_t} \seminorm{g_{\leq N}}_{L^{2m}_\omega C_{\mathrm{loc}}^{\alpha}C^{\vartheta_x}} + \seminorm{g_{N+1}}_{L^{2m}_\omega C_{\mathrm{loc}}^{\vartheta_t} C^{\vartheta_x}} \\
    & \lesssim \ell_{N+1}^{\alpha - \vartheta_t} \sum_{n=1}^N 4^{N-n}\seminorm{ g_n}_{L^{2m}_\omega C_{\mathrm{loc}}^{\alpha}C^{\vartheta_x}}  + \seminorm{g_{N+1}}_{L^{2m}_\omega C_{\mathrm{loc}}^{\vartheta_t} C^{\vartheta_x}} \\
    &\lesssim  \ell_{N+1}^{\alpha - \vartheta_t} \sum_{n=1}^N 4^{N-n} \lambda_{n}^{\vartheta_x - 1/2} \ell_n^{-\alpha} r_{n-1}^{1/2}  + \lambda_{N+1}^{\vartheta_x - 1/2} \ell_{N+1}^{-\vartheta_t} r_N^{1/2}.
\end{align*}
Since $\vartheta_t + \vartheta_x < 1/2 + \beta/(2b)$ the second term is controlled
\begin{align*}
    \lambda_{N+1}^{\vartheta_x - 1/2} \ell_{N+1}^{-\vartheta_t} r_N^{1/2} = a^{b^{N}(b(\vartheta_t + \vartheta_x - 1/2) - \beta/2)} r_{0}^{1/2}.
\end{align*}
We distinguish two cases for the estimate of the first term.

First let us consider $\vartheta_x \geq 1/2 - \alpha +  \beta/(2b)\Leftrightarrow b(\vartheta_x + \alpha - 1/2)  - \beta/2\geq 0$. Then estimating the super exponential by the last summand and bounding the geometric series 
\begin{align*}
     &\ell_{N+1}^{\alpha - \vartheta_t} \sum_{n=1}^N 4^{N-n} \lambda_{n}^{\vartheta_x - 1/2} \ell_n^{-\alpha} r_{n-1}^{1/2}  \\
       &\hspace{2em}= r_{0}^{1/2} a^{\beta/2} a^{b^{N+1}(\vartheta_t - \alpha)} \sum_{n=1}^N 4^{N-n} a^{b^{n-1}(b(\vartheta_x + \alpha - 1/2) - \beta/2)}  \\
       &\hspace{2em}\leq r_{0}^{1/2}a^{\beta/2} a^{b^{N-1}(b^2(\vartheta_t - \alpha)  + b(\vartheta_x + \alpha - 1/2) - \beta/2 )} 4^N.
\end{align*}
In order for the exponent to have a sign we need
\begin{align*}
\hspace{2em} & b^2(\vartheta_t - \alpha)  + b(\vartheta_x + \alpha - 1/2) - \beta/2  < 0 \\
\Leftrightarrow \hspace{2em} &   \vartheta_t     < \alpha + \frac{\beta - 2b(\vartheta_x + \alpha - 1/2)}{2b^2}.
\end{align*}
Thus, if $\vartheta_t$ and $\vartheta_x$ satisfy~\eqref{cond:HighSpace} we can estimate
\begin{align*}
    \seminorm{g_{\leq N+1} - g_{\leq N} }_{L^{2m}_\omega C_{\mathrm{loc}}^{\vartheta_t} C^{\vartheta_x}} \lesssim r_{0}^{1/2} a^{-N\varsigma_1},
\end{align*}
where $-\varsigma_1 = \max \{b(\vartheta_t + \vartheta_x - 1/2) - \beta/2,\,b^2(\vartheta_t - \alpha)  + b(\vartheta_x + \alpha - 1/2) - \beta/2\}$.

Now, let us consider the second case, i.e., $\vartheta_x < 1/2 - \alpha +  \beta/(2b)$. In this situation we can just neglect the super exponential in the sum and derive
\begin{align*}
     &\ell_{N+1}^{\alpha - \vartheta_t} \sum_{n=1}^N 4^{N-n} \lambda_{n}^{\vartheta_x - 1/2} \ell_n^{-\alpha} r_{n-1}^{1/2}  \\
       &\hspace{2em}= r_{0}^{1/2} a^{\beta/2} a^{b^{N+1}(\vartheta_t - \alpha) } \sum_{n=1}^N 4^{N-n} a^{b^{n-1}(b(\vartheta_x + \alpha - 1/2) - \beta/2)}  \\
       &\hspace{2em}\leq  r_{0}^{1/2} a^{\beta/2} a^{b^{N+1}(\vartheta_t - \alpha) } 4^N.
\end{align*}
Using~\eqref{cond:HighTime} it follows
\begin{align*}
    \seminorm{g_{\leq N+1} - g_{\leq N} }_{L^{2m}_\omega C_{\mathrm{loc}}^{\vartheta_t} C^{\vartheta_x}} \lesssim r_{0}^{1/2}a^{-N\varsigma_2},
\end{align*}
where $-\varsigma_2 = \max \{b(\vartheta_t + \vartheta_x - 1/2) - \beta/2,\, \vartheta_t - \alpha \}$.

The assertion~\eqref{eq:RegularityCor} follows, using the summability of $a^{-n\varsigma}$ (applied for $\varsigma \in \{\varsigma_1,\varsigma_2\}$) and $g_{\leq 0} = 0$,
\begin{align*}
     \seminorm{g}_{L^{2m}_\omega C_{\mathrm{loc}}^{\vartheta_t} C^{\vartheta_x}} \leq \sum_{n=0}^\infty  \seminorm{g_{\leq n+1} - g_{\leq n} }_{L^{2m}_\omega C_{\mathrm{loc}}^{\vartheta_t} C^{\vartheta_x}} \lesssim r_0^{1/2} \sum_{n=0}^\infty  a^{-n\varsigma} \lesssim r_0^{1/2}.
\end{align*}
\end{proof}

\begin{remark}
If we choose $\vartheta_t = 0$, we can pick $\vartheta_x \approx 1/2 + \beta/(2b) > 1/2$. On the other hand, if we choose $\vartheta_x = 1/2$, we can take $\vartheta_t \approx \min\{\alpha, \beta/(2b)\} > 0$. Therefore, $g$ enjoys improved regularity in the sense that
\begin{align}
   g \in L^{2m}_\omega \left( C_{\mathrm{loc}} C^{1/2 + \beta/(2b)-}_x \cap  C_{\mathrm{loc}}^{\min\{\alpha, \beta/(2b)\}-} B^{1/2}_{\infty,1} \right).
\end{align} 
\end{remark}

\section{Examples of noises $W$}

\label{examples section}
As seen in Theorem \ref{convergenceTheorem} and more precisely in \eqref{eq:NormEstimateLimit}, all our convex integration scheme requires is the finiteness of the term 
\begin{equation}
    \norm{z}_{L^{2m}_\omega C_{\mathrm{loc}} C^{1+\kappa}_x} + \seminorm{z }_{L^{2m}_\omega C^\alpha_{\mathrm{loc}} C^{1+}_x} 
    \label{normz}
\end{equation}
for some $\kappa>1/2$ and $\gamma<3/2$ and $\alpha>0$. Recalling that 
\[
z_t=\int_0^t e^{-\nu (t-s)\Lambda^\gamma}\Lambda^{\delta-1} {\rm d}W_s
\]
we provide some examples of noises $W$ and according constraints on the parameter $\delta$ which assure that \eqref{normz} is finite. For the sake of conciseness, we set the viscosity $\nu=1$ in this section.
\subsection{Cylindrical Wiener noise}\label{subsec:whiteNoise}
The canonical example is to take $W$ to be a cylindrical Wiener process. Let us stress that in our setting, the parameter $\delta$ will always be negative. The process $z$ should be accordingly thought of as the stochastic convolution of the colored noise $\tilde{W}:=\Lambda^{\delta-1}W$. We cite the following result from \cite{hofmanovaconvex6} on the regularity of the stochastic convolution in this setting.
\begin{lemma}[Regularity of stochastic convolution]
   Let $W$ be a cylindrical Wiener noise on $L^2(\mathbb{T}^2). $ Let $\gamma \in (0,3/2)$, $\kappa>1/2$ and $\delta< (-1-\kappa+\gamma/2)\wedge -1 $. Then for all $m\geq 1$ and $\alpha<1/2$ there exists $C> 1$  such that
    \begin{align}
        \seminorm{z}_{L^{2m}_\omega C^\alpha_t C^{1+}_x} + \norm{z}_{L^{2m}_\omega C_t C^{1+\kappa}_x} \leq C<\infty.
    \end{align}
    \label{regularity of stoch conv, space time white noise }
\end{lemma}
\begin{proof}
   Refer to \cite[Proposition 4.1]{hofmanovaconvex6}.
\end{proof}
Note that in this case, $z$ is a Gaussian random variable. As a direct consequence, from \eqref{eq:NormEstimateLimit} we conclude that the moments of our convex integration solution $g$ are sub-Gaussian. Indeed, by \eqref{eq:NormEstimateLimit}  we obtain
    \[
    \norm{g}_{L^{2m}_\omega C_{\rm loc} B^{1/2}_{\infty, 1}}\lesssim  \seminorm{z}_{L^{2m}_\omega C^\alpha_t C^{1+}_x} + \norm{z}_{L^{2m}_\omega C_t C^{1+\kappa}_x}  \lesssim C (2m)^{1/2}
    \]
where in the last step, we exploited that also the H\"older modulus of continuity of a Gaussian process admits sub-Gaussian moments (see for example \cite{hoelderexponential}). In particular, this quantified moment bound in $m$ implies that $g$ admits square expontential moments, i.e. there exists an $\epsilon>0$ such that 
\[
\mathbb{E}\left[\exp{\epsilon \norm{g}_{C_{\rm loc}B^{1/2}_{\infty, 1}}^2}\right]<\infty, 
\]
which is a result not available in the setting of \cite{hofmanovaconvex6}. 
\subsection{Noises constructed via Young integrals}
In this section, we show how to construct the stochastic convolution $z$ for a given noise $W$ via sewing techniques.  The latter constrains us to work on a compact time interval $[0, T]$, fixed throughout this subsection. We are basically able to cover the class of processes which admit some temporal H\"older regularity, provided the spatial coloring is strong enough, i.e. $\delta$ sufficiently negative. Let  $(\beta^k)_k$ be a sequence of independent copies of a stochastic process $\beta$ such that  
\begin{equation}
     \mathbb{E}\left[\left(\sup_{s\neq t\in [0, T]} \frac{|\beta_t-\beta_s|}{|t-s|^H}\right)^{2m}\right]<\infty.
     \label{hoelder reg noise}
\end{equation}
For $(e_k)_k$ an orthonormal basis of $L^2(\mathbb{T}^2)$, let us consider a noise of the form
\[
W_t=\sum_{|k|=0}^\infty e_k \beta^k_t.
\]
We then have the following
\begin{lemma}
\label{stoch conv via young}
    Let $W$ be as above. Suppose $\gamma<3/2$ and $\kappa>1/2$ and $H\in (0, 1)$. Suppose moreover $\delta<-2-(\gamma\vee \kappa+\gamma(1-H))$.  Then 
    \[
     \seminorm{z}_{L^2_\omega C^H_{[0, T]} C^{1+}_x} + \norm{z}_{L^2_\omega C_{[0, T]} C^{1+\kappa}_x}\lesssim  \mathbb{E}\left[ [\beta]^{2m}_{C^H}\right]^{\frac{1}{2m}}
    \]
\end{lemma}
\begin{example}
    A canonical example for $\beta$ in the above would be a $H+\epsilon$-fractional Brownian for $H\in (0, 1)$ and $\epsilon>0$ arbitrarily small. Indeed, in this case \eqref{hoelder reg noise} is finite and moreover also enjoys sub-Gaussian moment explosion. Similar to the above case of cylindrical Wiener noise,  this implies that our constructed convex integration solution $g$ enjoys square exponential moments in the sense that for some~$\epsilon'>0$, we have
\[
\mathbb{E}\left[\exp{\epsilon' \norm{g}_{C_{[0, T]}B^{1/2}_{\infty, 1}}^2}\right]<\infty.
\]
We stress that the construction is feasible for all $H\in (0, 1)$, as the integrand of the stochastic convolution is smooth in time. 
\end{example}
\begin{remark}
    Note that in the case $H=1/2$, which recovers the case of cylindrical Wiener noise, the condition in the above lemma is worse than that given in Lemma \ref{regularity of stoch conv, space time white noise }. Besides the fact that we loose optimality due to Sobolev embeddings, this would still be natural to expect, as the sewing arguments are entirely pathwise and thus do not capture any stochastic cancellations. 
\end{remark}

\begin{proof}[Proof of Lemma \ref{stoch conv via young}]
    Let $(e_k)_k$ be the sequence of eigenvectors of $\Lambda$ and $(\lambda_k)_k$ be the corresponding sequence of eigenvalues. Note that $\lambda_k\simeq |k|$. We begin by showing that the expression
    \begin{equation}
         \int_0^t e^{-\lambda_k^\gamma(t-r)}{\rm d}\beta_r
         \label{sewing start}
    \end{equation} 
    is well defined as the sewing of the germ $A^t_{s,r}=e^{-\lambda_k^\gamma(t-s)}(\beta_r-\beta_s)$. For the reader unfamiliar with sewing techniques and the corresponding terminologies, we refer to section \ref{sewing section} in the appendix. Let us point out that $A^t:\Delta_2([0, t))\to \mathbb{R}$ is indeed a well defined object for any $t\in [0, T]$, since - contrary to Volterra settings - no singularities in $t\in [0, T]$ are present.  Using the Schauder estimate Lemma \ref{schauder}, we note that $0<\epsilon<H$
    \[
    |(\delta A^t)_{s, u, r}|=|(\beta_r-\beta_u)e^{-\lambda_k^\gamma t}(e^{\lambda_k^\gamma s}-e^{\lambda_k^\gamma u})|\lesssim [\beta]_{C^H}|r-u|^H(\lambda_k^\gamma)^{1-H+\epsilon}|u-s|^{1-H+\epsilon}.
    \]
    By the Sewing Lemma \ref{sewing}, we therefore obtain the existence of an associated integral $(\mathcal{I}A^t)$, which we use to define \eqref{sewing start}
    \[
    (\mathcal{I}A^t)_t=:\int_0^t e^{-\lambda_k^\gamma(t-r)}{\rm d}\beta_r.
    \]
Using the a priori bound \eqref{sewing a priori} that comes with the Sewing Lemma \ref{sewing}, we have 
    \[
       |(\mathcal{I}A^t)_t-   (\mathcal{I}A^t)_s|\lesssim |t-s|^H[\beta]_{C^H}+|t-s|^{1+\epsilon}[\beta]_{C^H}(\lambda_k^\gamma)^{1-H+\epsilon}.
    \]
    With the definition of the one dimensional stochastic convolution at hand, we define
    \[
    z_t:=\sum_{|k|=1}^\infty \lambda_k^{\delta-1}\int_0^t e^{-\lambda_k^\gamma(t-s)}{\rm d}\beta^k_s e_k
    \]
  We note that
    \begin{equation*}
        \begin{split}
                \norm{z_t-z_s}_{H^\sigma}^2&=\sum_{|k|=1}^\infty \lambda_k^{2(\delta-1)} |k|^{2\sigma} \left(\int_0^t e^{-\lambda_k^\gamma(t-r)}{\rm d}\beta^k_r- \int_0^s e^{-\lambda_k^\gamma(s-r)}{\rm d}\beta^k_r\right)^2\\
                &\lesssim \sum_{|k|=1}^\infty \lambda_k^{2(\delta-1)} |k|^{2\sigma} \left(\int_s^t e^{-\lambda_k^\gamma(t-r)}{\rm d}\beta^k_r\right)^2\\
                &\quad+\sum_{|k|=1}^\infty \lambda_k^{2(\delta-1)} |k|^{2\sigma} (e^{-\lambda_k^\gamma(t-s)}-1)^2\left(\int_0^s e^{-\lambda_k^\gamma(s-r)}{\rm d}\beta^k_r\right)^2
        \end{split}
    \end{equation*}
    We can now use the above estimate to deduce that 
    \[
    \left|\int_s^t e^{-\lambda_k^\gamma(t-r)}{\rm d}\beta^k_r\right|=|(\mathcal{I}A^t)_t-   (\mathcal{I}A^t)_s|\lesssim |t-s|^H[\beta^k]_{C^H}+|t-s|^{1+\epsilon}[\beta^k]_{C^H}(\lambda_k^\gamma)^{1-H+\epsilon}
    \]
    and 
    \[
\left|\int_0^s e^{-\lambda_k^\gamma(s-r)}{\rm d}\beta^k_r\right|=|(\mathcal{I}A^s)_s-   (\mathcal{I}A^s)_0|\lesssim T^H[\beta^k]_{C^H}+T^{1+\epsilon}[\beta^k]_{C^H}(\lambda_k^\gamma)^{1-H+\epsilon}.
    \]
    We therefore conclude that 
    \begin{equation*}
        \begin{split}
            &\mathbb{E}\left[\sum_{|k|=1}^\infty \lambda_k^{2(\delta-1)} |k|^{2\sigma} \left(\int_s^t e^{-\lambda_k^\gamma(t-r)}{\rm d}\beta^k_r\right)^2\right]\\
            &\lesssim   \mathbb{E}\left[[\beta]_{C^H}^2\right]|t-s|^{2H}\sum_{|k|=1}^\infty \lambda_k^{2(\delta-1)} |k|^{2\sigma}(\lambda_k^\gamma)^{2(1-H+\epsilon)},
        \end{split}
    \end{equation*}
    where the above sum converges provided 
    \begin{equation}
        \label{condition 1}
            2(\delta-1)+2\sigma+2\gamma(1-H+\epsilon)<-2.
    \end{equation}
    Concerning the second sum, we obtain again thanks to Schauder estimates $|e^{-\lambda_k^\gamma(t-s)}-1|\lesssim \lambda_k^{\gamma H}|t-s|^H$ (Lemma \ref{schauder}) and thus 
     \begin{equation*}
        \begin{split}
 & \mathbb{E}\left[\sum_{|k|=1}^\infty \lambda_k^{2(\delta-1)} |k|^{2\sigma} (e^{-\lambda_k^\gamma(t-s)}-1)^2\left(\int_0^s e^{-\lambda_k^\gamma(s-r)}{\rm d}\beta^k_r\right)^2\right]\\
 &\lesssim \mathbb{E}\left[[\beta]_{C^H}^2\right] \sum_{|k|=1}^\infty \lambda_k^{2(\delta-1)} |k|^{2\sigma} (\lambda_k^\gamma)^{2H+2(1-H+\epsilon)}|t-s|^{2H}
      \end{split}
    \end{equation*}
    where the above sum converges provided
    \begin{equation}
        \label{condition 2}
        2(\delta-1)+2\sigma+2\gamma(1+\epsilon)<-2.
    \end{equation}

Noting that \eqref{condition 2} implies \eqref{condition 1}, we can conclude that $z\in L^2_\omega C^H_tH^\sigma_x$ provided \eqref{condition 2} holds. Using the Sobolev embedding $H^\sigma\hookrightarrow C^{1+\epsilon}$ for $\sigma>2$ and $\epsilon>0$ sufficiently small, we have $z\in L^2_\omega C^H_tC^{1+\epsilon}_x$. Concerning the other regularity claim on $z$, note that it suffices to ask for arbitrarily small H\"older regularity, meaning that in this case we may use the Schauder estimate $|e^{-\lambda_k^\gamma(t-s)}-1|\lesssim \lambda_k^{\gamma \epsilon}|t-s|^\epsilon$, which lets \eqref{condition 2} collapse to \eqref{condition 1}. We therefore obtain that $z\in L^2_\omega C^\epsilon_t H^\sigma_x$ for some $\epsilon>0$, provided \eqref{condition 1} holds. Using again the Sobolev embedding $H^{\sigma}\hookrightarrow C^{1+\kappa+\epsilon}$ for $\sigma>2+\kappa$, we obtain the claim. 
\end{proof}

\subsection{Another class of general stochastic processes}
We give another example of a class of stochastic processes $W$ that assures the existence of a stochastic convolution $z$ and that can assure the finiteness of \eqref{normz}. Contrary to the previous subsection, we will provide a construction in spirit of that of the It\^o integral (i.e. in particular not a pathwise one), yet going beyond the martingale and Gaussian setting.\\
\\
Let $\beta:\Omega \times [0, \infty)\to \mathbb{R}$  be a stochastic process which starts in zero, is continuous and centered and has independent increments. We moreover assume that 
\begin{equation}
    \label{moment assumptions}
    \mbox{Var}(\beta_t-\beta_s)=(t-s), \qquad \mathbb{E}[|\beta_t-\beta_s|^{4}]\leq C|t-s|^{1+\upsilon}
\end{equation}
for some $\upsilon\in (0,1]$. Note that the above are obviously satisfied for $\beta$ being a Brownian motion with $\upsilon=1$, however $\beta$ need not be a Brownian motion or even a Gaussian process (if it was a Gaussian process, the above conditions would imply that $\beta$ is a Brownian motion). 
We consider again
\[
W_t=\sum_{k}e_k \beta^k_t
\]
where $(e_k)_k$ is an orthonormal basis of $L^2(\mathbb{T}^2)$ and $(\beta^k)_k$ are independent copies of $\beta$, which we describe below. In the case of deterministic integrands which we are concerned with, the above conditions on $\beta$ suffice to establish an analogue of the Paley-Wiener integral by means of an isometry. Time regularity of the so obtained integral can then be recovered by means of  the Kolmogorov continuity theorem (which is why we require a fourth moment assumption). 
\begin{lemma}
    Let $f:[0, t]\to \mathbb{R}$ be a deterministic elementary function meaning for some $N\in \mathbb{N}$ and $0\leq t_1< \dots t_N\leq t$ and $c_k\in \mathbb{R}$, we have
    \[
    f(t)=\sum_{k=0}^{N-1} c_k \mathbf{1}_{[t_k, t_{k+1}]}.
    \]
    Suppose $\beta$ as above, i.e. a continuous centered process starting in zero satisfying \eqref{moment assumptions}. Then we have 
    \[
    \mathbb{E}\left(\int_0^t f_s {\rm d}\beta_s \right)^2= \int_0^t f_s^2{\rm d}s
    \]
\end{lemma}
\begin{proof}
    The proof follows from a direct computation as in the case of the It\^o isometry, exploiting that the coefficients $(c_k)_{k\leq N}$ are deterministic. 
\end{proof}
As the elementary functions are dense in $L^2([0, t])$, we can use the above isometry to define the Paley-Wiener integral for $f\in L^2([0, t])$. Indeed, if $f\in L^2([0, t])$, we find a sequence of elementary function $(f^n)_n$ such that $\norm{f^n-f}_{L^2}\to 0$. Thus 
\[
\mathbb{E}\left(\int_0^t f^n_s-f^m_s {\rm d}\beta_s \right)^2= \int_0^t (f^n_s-f^m_s)^2{\rm d}s
\]
i.e. the $\left(\int_0^t f^n_s{\rm d}\beta_s \right)$ is Cauchy in $L^2(\Omega)$, meaning we obtain a limit object for which the above isometry carries over. In particular, we have 
\[
 \mathbb{E}\left(\int_0^t e^{-\lambda(t-s)} {\rm d}\beta_s \right)^2= \int_0^t e^{-2\lambda(t-s)}{\rm d}s=\frac{1}{2\lambda}(1-e^{-\lambda t}).
\]
Moreover, we can use the estimate on the fourth moment in \eqref{moment assumptions} to deduce the existence of a H\"older continuous modification of the integral process via Kolmogorov's continuity theorem.

\begin{lemma}
    Let $f:[0, T]\to \mathbb{R}$ be bounded. We have 
    \[
    \mathbb{E}\left(\int_s^tf_r{\rm d}\beta_r \right)^4\leq 6\left(\int_s^t f^2_r{\rm d}r\right)^2+\norm{f}_\infty^{4-2(1+\upsilon)}\left(\int_s^t f_r^2{\rm d}r\right)^{1+\upsilon}.
    \]
    Consequently, the process $X_t=\int_0^tf_s{\rm d}\beta_s$ admits a $\frac{\upsilon}{4}$-H\"older continuous modification. 
    \label{hoelder stoch int}
\end{lemma}

\begin{proof}
   By density, we may argue again through approximation via elementary functions: let 
    \[
    f^N_t=\sum_{k=1}^N c_k \mathbf{1}_{[t_k, t_{k+1}]}(t).
    \]
We have 
\begin{equation*}
    \begin{split}
            &\mathbb{E}\left(\int_s^tf^N_r{\rm d}\beta_r \right)^4 \\
            &=\sum_{k, j, n, m=1}^Nc_kc_jc_nc_m \mathbb{E}(\beta_{t_{k+1}}-\beta_{t_k})(\beta_{t_{j+1}}-\beta_{t_j})(\beta_{t_{n+1}}-\beta_{t_n})(\beta_{t_{m+1}}-\beta_{t_m}).
    \end{split}
\end{equation*}
 Observe that due to the independent increments and centeredness, the above expectation vanishes whenever all four running indices $k, j, n, m$ are distinct and whenever three of them coincide but one is distinct. We are therefore left with 
 \begin{equation*}
    \begin{split}
            \mathbb{E}\left(\int_s^tf^N_r{\rm d}\beta_r \right)^4 &=6\sum_{k\neq j}^Nc_k^2c_j^2 \mathbb{E}(\beta_{t_{k+1}}-\beta_{t_k})^2(\beta_{t_{j+1}}-\beta_{t_j})^2+\sum_{k=1}^Nc_k^4\mathbb{E}(\beta_{t_{k+1}}-\beta_{t_k})^4\\
            &\leq 6\sum_{k=1}^Nc_k^2 (t_{k+1}-t_k)\sum_{j=1}^N c_j (t_{j+1}-t_j)+\sum_{k=1}^N c_k^4 (t_{k+1}-t_k)^{1+\upsilon}\\
            &\leq 6\left(\int_s^t(f_r^N)^2dr\right)^2+\norm{f^N}_{\infty}^{4-2(1+\upsilon)}\sum_{k=1}^N(c_k^2(t_{k+1}-t_k))^{1+\upsilon}\\
            &\leq  6\left(\int_s^t(f_r^N)^2dr\right)^2+\norm{f^N}_{\infty}^{4-2(1+\upsilon)}\left(\int_s^t(f_r^N)^2dr\right)^{1+\upsilon}.
    \end{split}
\end{equation*}
We then conclude by density and Kolmogorov's continuity theorem.
\end{proof}
From the above proof, it is clear that the statement extends to deterministic bounded functions depending also on the upper integration variable, that is for $f:\Delta_2\to \mathbb{R}$ bounded deterministic, we have 
\[
    \mathbb{E}\left(\int_s^tf(t,r){\rm d}\beta_r \right)^4\leq 6\left(\int_s^t f^2(t, r){\rm d}r\right)^2+\norm{f}_\infty^{4-2(1+\upsilon)}\left(\int_s^t f^2(t,r){\rm d}r\right)^{1+\upsilon}.
    \]
    Consequently, $|\int_s^tf(t,r){\rm d}\beta_r|\lesssim |t-s|^{\upsilon/4}$ $\mathbb{P}$-almost surely. This essentially gives us H\"older regularity of the stochastic integral. To conclude that the stochastic convolution is H\"older, we need to also obtain some H\"older regularity from the semigroup, for which we use again the Schauder estimate of Lemma \ref{schauder}. Indeed, for $z_t=\int_0^te^{-a(t-s)}{\rm d}\beta_s$ we have
\[
z_t-z_s=\int_s^t e^{-a(t-r)}{\rm d}\beta_r+\int_0^s (e^{-a(t-r)}-e^{-a(s-r)}){\rm d}\beta_s
\]
and therefore using Lemma \ref{hoelder stoch int} and the Schauder estimate \ref{schauder} (using it once with $\theta=(1+\upsilon)/4$ and once with $\theta=1/2$), we obtain
\begin{equation}
\label{fourth moment z}
    \begin{split}
        \mathbb{E}|z_t-z_s|^4&\lesssim (t-s)^2+(t-s)^{1+\upsilon}\\
        &+\left(\int_0^s (e^{-a(t-r)}-e^{-a(s-r)})^2{\rm d}r\right)^2+\left(\int_0^s (e^{-a(t-r)}-e^{-a(s-r)})^2{\rm d}r\right)^{1+\upsilon}\\
        &\lesssim (t-s)^2+(t-s)^{1+\upsilon}+a^{(1+\upsilon)}(t-s)^{(1+\upsilon)}.
    \end{split}
\end{equation}
Hence, we can apply the Kolmogorov continuity theorem to obtain the existence of a continuous modification. In the infinite dimensional setting, we need to again make sure that we stay summable along Fourier modes but the calculation is essentially similar. 
\begin{lemma}
\label{regularity stoch conv 4th moment}
    Let $\beta:\Omega \times [0, \infty)\to \mathbb{R}$ be a centered, continuous stochastic process starting in zero satisfying \eqref{moment assumptions}. Consider as noise 
\[
W_t=\sum_k e_k \beta^k_t
\]
where $(e_k)$ is an orthonormal basis of $L^2(\mathbb{T}^2)$ of eigenvectors of $\Lambda$ and $(\beta^k)_k$ are independent copies of $\beta$. Then the process
\begin{equation}
    z_t:=\int_0^t e^{\Lambda^\gamma(t-s)}\Lambda^{\delta-1} {\rm d}W_s=\sum_{|k|=1}^\infty \int_0^t e^{-\lambda_k^\gamma(t-s)}\lambda_k^{\delta-1} {\rm d}\beta^k_s e_k
    \label{stoch conv 33}
\end{equation}
is well defined and we have for any $\epsilon>0$
\[
z\in L^4_\omega C^{\upsilon/4-\epsilon}_{[0, T]} H^\sigma_x
\]
provided the parameters satisfy 
\[
  4(\delta-1)+4\sigma+(1+\upsilon)\gamma<-2.
\]
\end{lemma}
Using again Sobolev embeddings, the above implies:
\begin{corollary}
   Let $\gamma\in (0, 3/2)$ and $\kappa>1/2$. Suppose that $\upsilon$ as above and 
   \[
   \delta<-3/2-(1+\upsilon)\gamma/4-\kappa.
   \]
   Then the process $z$ in \eqref{stoch conv 33} satisfies $z\in L^4_{\omega} C^{\upsilon/4-\epsilon}_{[0, T]} C^{1+\kappa}_x$. 
\end{corollary}

\begin{proof}[Proof of Lemma \ref{regularity stoch conv 4th moment}]
    Note that 
    \begin{equation*}
        \begin{split}
            \norm{z_t-z_s}_{H^\sigma}^2=\sum_{|k|=1}^\infty \lambda_k^{2(\delta-1)} |k|^{2\sigma} \left(\int_0^t e^{-\lambda_k^\gamma(t-r)}{\rm d}\beta^k_r- \int_0^s e^{-\lambda_k^\gamma(s-r)}{\rm d}\beta^k_r\right)^2
        \end{split}
    \end{equation*}
    and thus 
      \begin{equation*}
        \begin{split}
            &\mathbb{E}\norm{z_t-z_s}_{H^\sigma}^4\\
            &=\sum_{k\neq j}^\infty \lambda_k^{2(\delta-1)} |k|^{2\sigma} \mathbb{E}\left(\int_0^t e^{-\lambda_k^\gamma(t-r)}{\rm d}\beta^k_r- \int_0^s e^{-\lambda_k^\gamma(s-r)}{\rm d}\beta^k_r\right)^2\\
            &\qquad\times\lambda_j^{2(\delta-1)} |j|^{2\sigma} \mathbb{E}\left(\int_0^t e^{-\lambda_j^\gamma(t-r)}{\rm d}\beta^j_r- \int_0^s e^{-\lambda_j^\gamma(s-r)}{\rm d}\beta^j_r\right)^2\\
            &\quad+\sum_{|k|=1}^\infty \lambda_k^{4(\delta-1)} |k|^{4\sigma} \mathbb{E}\left(\int_0^t e^{-\lambda_k^\gamma(t-r)}{\rm d}\beta^k_r- \int_0^s e^{-\lambda_k^\gamma(s-r)}{\rm d}\beta^k_r\right)^4\\
            &\lesssim \left(\sum_{|k|=1}^\infty \lambda_k^{2(\delta-1)}|k|^{2\sigma}\mathbb{E}\left(\int_0^t e^{-\lambda_k^\gamma(t-r)}{\rm d}\beta^k_r- \int_0^s e^{-\lambda_k^\gamma(s-r)}{\rm d}\beta^k_r\right)^2\right)^2\\
            &\quad+\sum_{|k|=1}^\infty \lambda_k^{4(\delta-1)} |k|^{4\sigma} \mathbb{E}\left(\int_0^t e^{-\lambda_k^\gamma(t-r)}{\rm d}\beta^k_r- \int_0^s e^{-\lambda_k^\gamma(s-r)}{\rm d}\beta^k_r\right)^4
        \end{split}
    \end{equation*}
    Using \eqref{fourth moment z}, we have 
    \begin{equation*}
        \begin{split}
               &\sum_{|k|=1}^\infty \lambda_k^{4(\delta-1)} |k|^{4\sigma} \mathbb{E}\left(\int_0^t e^{-\lambda_k^\gamma(t-r)}{\rm d}\beta^k_r- \int_0^s e^{-\lambda_k^\gamma(s-r)}{\rm d}\beta^k_r\right)^4\\
    &\lesssim \sum_{|k|=1}^\infty \lambda_k^{4(\delta-1)} |k|^{4\sigma} \left(|t-s|^2+\lambda_k^{\gamma (1+\upsilon) }(t-s)^{1+\upsilon} \right).
        \end{split}
    \end{equation*}
The above sum converges if
\begin{equation}
    4(\delta-1)+4\sigma+(1+\upsilon)\gamma<-2
    \label{condition 3}
\end{equation}
in which case the above sum is bounded by 
\[
\left|\sum_{|k|=1}^\infty \lambda_k^{4(\delta-1)} |k|^{4\sigma} \mathbb{E}\left(\int_0^t e^{-\lambda_k^\gamma(t-r)}{\rm d}\beta^k_r- \int_0^s e^{-\lambda_k^\gamma(s-r)}{\rm d}\beta^k_r\right)^4\right|\lesssim |t-s|^{1+\upsilon}.
\]
Concerning the second sum, note that we can bound similarly using again \eqref{fourth moment z}
\begin{equation*}
\begin{split}
     &\sum_{|k|=1}^\infty \lambda_k^{2(\delta-1)}|k|^{2\sigma}\mathbb{E}\left(\int_0^t e^{-\lambda_k^\gamma(t-r)}{\rm d}\beta^k_r- \int_0^s e^{-\lambda_k^\gamma(s-r)}{\rm d}\beta^k_r\right)^2\\
    &\lesssim\sum_{|k|=1}^\infty \lambda_k^{2(\delta-1)}|k|^{2\sigma}\left(\mathbb{E}\left(\int_0^t e^{-\lambda_k^\gamma(t-r)}{\rm d}\beta^k_r- \int_0^s e^{-\lambda_k^\gamma(s-r)}{\rm d}\beta^k_r\right)^4\right)^{1/2}\\
    &\lesssim\sum_{|k|=1}^\infty \lambda_k^{2(\delta-1)}|k|^{2\sigma}\left(|t-s|^2+\lambda_k^{\gamma (1+\upsilon) }(t-s)^{1+\upsilon} \right)^{1/2}.
\end{split}
\end{equation*}
Under condition \eqref{condition 3}, the above sum converges in which case we obtain
\begin{equation*}
    \begin{split}
        \left(\sum_{|k|=1}^\infty \lambda_k^{2(\delta-1)}|k|^{2\sigma}\mathbb{E}\left(\int_0^t e^{-\lambda_k^\gamma(t-r)}{\rm d}\beta^k_r- \int_0^s e^{-\lambda_k^\gamma(s-r)}{\rm d}\beta^k_r\right)^2\right)^2\lesssim |t-s|^{1+\upsilon}.
    \end{split}
\end{equation*}
Overall, this shows
\[
\mathbb{E}\norm{z_t-z_s}_{H^\sigma}^4\lesssim |t-s|^{1+\upsilon}
\]
and thus, we may conclude by Kolmogorov's continuity theorem. 
\end{proof}

\section{Consequences}
\label{consequences section}
\noindent
In this section, we collect a selection of consequences of Section \ref{convex int section}. By nature of the presented estimates, the following results directly carry over from the analysis in \cite{hofmanovaconvex4} and \cite{hofmanovaconvex6}, hence we will content ourselves with merely giving the statements and for the detailed proofs refer the reader to the corresponding arguments in those papers. We would like to stress here that we do not require Gaussianity of the noise, hence the results of these sections hold for any noise $W$ such that the associated stochastic convolution $z$ is a well defined object satisfying \eqref{normz}. However, in order to define stationary (ergodic) solutions, Section \ref{sec:statSol} may only deal with noises allowing for an extension of the corresponding solution to the full time horizon.

\subsection{Non-uniqueness, infinitely many solutions and continuum of solutions}\label{sec:nonunique}
\begin{theorem}
    The solution constructed in Theorem \ref{convergenceTheorem} is not unique; in fact, there exist infinitely many such solutions and the set of solutions forms a continuum.
\end{theorem}
These statements follow analogously to the analysis in Sections 3.2.2, 3.2.3, and 3.2.4, respectively, in \cite{hofmanovaconvex4}.

\subsection{Initial value problem}\label{sec:intincond}

Let $\theta_0$ be an $\mathcal{F}_0$-measurable initial condition independent of the noise with zero mean and such that for all $m \geq 1$ and some $\eta > \frac{1}{2}$ it holds
\begin{equation}\label{momentInitCond}
\mathbb{E}\left[ \|\theta_0\|_{C^{\eta}}^{2m}\right] < \infty.
\end{equation}
Similar to \cite{hofmanovaconvex6} we may incorporate this initial condition into the stochastic convolution giving
\[ z_t =  e^{-t\Lambda^{\gamma}} \Lambda^{-1} \theta_0 + \Lambda^{-1}\int_0^t e^{-(t-s) \nu \Lambda^{\gamma}}\Lambda^{\delta-1} {\rm d}W_s\]
where the first summand is in $C([0,\infty); C^{\eta+1}) \cap C_t^{1/2}C_x^{\eta+1-\gamma/2}$, cf. \cite[Lemma 2.8]{zhang2022}. In the convex integration construction of Section \ref{convex int section}, adjust by introducing the cut-off
\begin{equation*}
	\chi(t) =
\begin{cases}
0,& t \leq 2^{-n-1}\\
\in (0,1),& t \in \left(2^{-n-1},2^{-n}\right)\\
1,& t \geq 2^{-n}
\end{cases}
\end{equation*}
and redefining the perturbations as
\[ g_{n+1} := \chi g_{n+1}^{\rm old}.\]
This way, the scheme does not manipulate the initial condition and though we merely obtain a bounded error
\[ \left(\mathbb{E}\left[\sup_{t \in [0, 2^{-n}]} \|q_{n+1}(t)\|_X^r\right]\right)^{\frac{1}{r}} \lesssim C_{\rm start}(C_z^2 +1)\]
in that region, this is compensated by the decreasing size of the time interval. In contrast to Section \ref{convex int section}, the constant $C_{\rm start}$ will now depend on the initial condition; more precisely it depends on the moment bound \ref{momentInitCond}. Therefore in this setting, one first needs to fix $m$ to proceed with the analysis; nevertheless, the result holds for all $m \geq 1$. Thus in analogy to \cite{hofmanovaconvex6} we obtain
\begin{theorem}\label{thm:initCond}
For the above choice of initial condition $\theta_0$ there exist infinitely many $(\mathcal{F}_t)$-adapted non-Gaussian analytically weak solutions $\theta$ to \eqref{sqg intro} in the sense of Lemma \ref{lem:trafo} with initial value $\theta_0$, belonging to
\[L^p_{\rm loc}([0,\infty); B_{\infty, 1}^{-1/2}) \cap C([0,\infty), B_{\infty,1}^{-1/2 - }) \cap C^{1/2 -}([0,\infty), B_{\infty,1}^{-1 -})\]
$\mathbb{P}$-a.s. for all $p \in [1,\infty)$. Furthermore there exists $\sigma > \frac{1}{2}$ such that for all $m \geq 1$
\begin{align*}
    &\sup_{t \geq 4} \left(\mathbb{E}\left[\sup_{s \in [t,t+1]} \|\theta\|_{B_{\infty,1}^{\sigma -1}}^{2m} \right]\right)^{\frac{1}{2m}} + \|\theta\|_{L_{\omega}^{2m}C_{\rm loc}^{1/2-}B_{\infty,1}^{-1-}} \\
    &\hspace{5em}\lesssim \sqrt{C_{\mathrm{start}}}(\norm{z}_{L^{2m}_\omega C_{\mathrm{loc}} C^{1+\kappa}_x} + \seminorm{z }_{L^{2m}_\omega C^\alpha_{\mathrm{loc}} C^{1+}_x} + 1).
\end{align*}
In particular, for any $\epsilon >0$ there exists such $\theta$ such that
\begin{equation}\label{eq:epsInitCond}
    \|\theta\|_{L_{\omega}^{2m}C_{\rm loc}B_{\infty,1}^{-1/2-}} \leq \sqrt{C_{\mathrm{start}}}(\norm{z}_{L^{2m}_\omega C_{\mathrm{loc}} C^{1+\kappa}_x} + \seminorm{z }_{L^{2m}_\omega C^\alpha_{\mathrm{loc}} C^{1+}_x} + 1) \left(\|\theta_0\|_{L^{2m}_{\omega}B_{\infty,1}^{-1/2-}} + \epsilon\right).
\end{equation}
\end{theorem}
\begin{proof}
	First observe that analogous to \cite[(6.2)]{hofmanovaconvex6} it holds 
	\begin{align*}
		\|g_{\ell_{n+1}} - g_{\leq n}\|_{L_{\omega}^{2m}C_{\rm loc}C^{\sigma}_x} &\lesssim \ell_{n+1}^{\alpha} \lambda_n^{\sigma -\frac{1}{2}}\|g_{\leq n}\|_{L_{\omega}^{2m}C_{\rm loc}^{\alpha}B_{\infty, 1}^{1/2}}  \\
  &\lesssim \sqrt{C_{\rm start}(1+ C_z^2)} \lambda_{n+1}^{-\alpha} \lambda_n^{\sigma} \leq \frac{1}{4} r_{n+1}^{\frac{1}{2}}
	\end{align*}
	for $\sigma < \frac{b}{2}(\alpha-\beta)$ and hence
	\begin{align*}
		&\sup_{t \geq 2^{-n+2}} \left(\mathbb{E}\left[\sup_{s \in [t,t+1]} \|g_{\leq n+1}(s) - g_{\leq n}(s)\|^{2m}_{C^{\sigma}}\right]\right)^{\frac{1}{2m}}\\
		&\leq \|g_{\ell_{n+1}} - g_{\leq n}\|_{L_{\omega}^{2m}C_{\rm loc}C^{\sigma}_x} + \sup_{t \geq 2^{-n+2}}\left(\mathbb{E}\left[\sup_{s \in [t,t+1]} \|g_{n+1}\|_{C^{\sigma}}^{2m}\right]\right)^{\frac{1}{2m}}\\
		&\lesssim r_{n+1}^{\frac{1}{2}} + \lambda_{n+1}^{\sigma - \frac{1}{2}}\sqrt{(C_0 +1) \left(\|q_n\|_{L_{\omega}^mC_{\rm loc}X} + r_n\right)}\\
		&\lesssim r_{n+1}^{\frac{1}{2}}+ \lambda_{n+1}^{\sigma - \frac{1}{2}} r_n^{\frac{1}{2}}
	\end{align*}
	where 
	\[ \lambda_{n+1}^{\sigma - \frac{1}{2}} r_n^{\frac{1}{2}} = \sqrt{C_{\rm start}(C_z^2+1)} \lambda_0^{\frac{\beta}{2}} \lambda_{n+1}^{\sigma - \frac{1}{2} - \frac{\beta}{2b}} \overset{!}{\leq} \sqrt{C_{\rm start}(C_z^2+1)} \lambda_0^{\frac{\beta}{2}} \lambda_{n+1}^{\frac{\beta}{4b}}\]
	for $\sigma < \frac{1}{2} + \frac{\beta}{4b}$. In total this gives
		\[ \left(\sup_{t \geq 2^{-n+2}} \mathbb{E}\left[\sup_{s \in [t,t+1]} \|g_{\leq n+1}(s) - g_{\leq n}(s)\|_{C^{\sigma}}^{2m}\right]\right)^{\frac{1}{2m}} \leq \sqrt{C_{\rm start}(C_z^2+1)}\lambda_0^{\frac{\beta}{2}} \lambda_{n+1}^{-\frac{\beta}{4b}} + r_{n+1}^{\frac{1}{2}}.\]
		Conclude as in \cite[Theorem 7.1]{hofmanovaconvex6}, and \cite[Theorem 3.2]{hofmanovaconvex4}. In particular, non-Gaussianity follows analogously: for non-Gaussian $\theta_0$ the parameter $a$ can be chosen large enough so as to ensure that $P_{\leq \lambda_0}\theta_0$ and $P_{\lambda_0} \hat{\varphi}$ are non-Gaussian and conclude via the construction. For Gaussian $\theta_0$, start the construction from $g_{\leq 0} = \hat{\chi}Z$ where $Z$ is a smooth bounded non-Gaussian $\mathcal{F}_0$-measurable random variable with finitely many non-zero Fourier modes, independent of $\theta_0$ and the noise, and $\hat{\chi}$ a smooth cut-off in time equal to 0 if $t \leq \frac{1}{2}$ and 1 if $t >1$. Choosing $\lambda_0$ large enough such that $P_{\leq \lambda_0}Z =Z$, we hence ensure that the first Fourier modes of the constructed solution correspond to that of a non-Gaussian and is thus non-Gaussian itself.
\end{proof}

\subsection{Prescribing terminal values}\label{sec:terminalValue}
As it turns out, a similar procedure as in Section \ref{sec:intincond} can be used to prescribe terminal values at arbitrary times $T \geq 4$:
\begin{theorem}\label{thm:presTerminal}
    Let $T \geq 4$ be arbitrary. For all $m\geq 1$, let $v \in L_{\omega}^mC^1([0,T],C^{\eta})$, $\eta > \frac{1}{2}$ be given such that $v$ zero mean $\mathbb{P}$-almost surely, $v(t) = v(T)$ for all $t \in [T-1,T]$ and $\partial_t v(0) =0$. Then there exists an analytically weak solution $\theta$ to \eqref{sqg intro} in the sense of Lemaa \ref{lem:trafo} such that $\theta(0)=v(0)$ and $\theta(T) =v(T)$ and
    \[\theta \in L^p(0,T; B_{\infty, 1}^{-1/2}) \cap C([0,T], B_{\infty,1}^{-1/2 - }) \cap C^{1/2 -}([0,T], B_{\infty,1}^{-1 -})\]
$\mathbb{P}$-a.s. for all $p \in [1,\infty)$. Furthermore for any $\epsilon >0$ there exists such $\theta$ such that for all $m\geq 1$
\[\|\theta\|_{L_{\omega}^{2m}C_TB_{\infty,1}^{-1/2-}} \leq \|v\|_{L_{\omega}^{2m}C_TB_{\infty,1}^{-1/2-}} + \epsilon.\]
\end{theorem}
\begin{proof}
    Analogously to \cite[Section 3.3]{hofmanovaconvex4} we may in this case set up a convex integration procedure of the form: let $z(0) =0$ and let
\begin{align*}
    &-\partial_t \mathcal{R}g_{\leq n} + \nabla^{\perp}(g_{\leq n} + z_n)\Lambda(g_{\leq n}+z_n) - \nu \Lambda^{\gamma-1}\nabla g_{\leq n} \mathring{\approx} \nabla q_n,\\
    &g_{\leq n}(0) = P_{\lambda_n/3}\Lambda^{-1}v(0),\\
    &g_{\leq n}(t) = P_{\lambda_n/3}\Lambda^{-1}v(T), \quad\forall t \in [T-2^{-n},T].
\end{align*}
Indeed, starting of with
\begin{align*}
    &g_{\leq 0} = P_{\lambda_0/3}\Lambda^{-1}v,\\
    &\nabla q_0 \mathring{\approx} -\partial_t \mathcal{R}g_{\leq 0} + \nabla^{\perp}(g_{\leq 0} + z_0)\Lambda(g_{\leq 0}+z_0) - \nu \Lambda^{\gamma-1}\nabla g_{\leq 0}
\end{align*}
gives the estimates
\begin{align*}
    \|g_{\leq 0} \|_{L^{2m}_{\omega}C_T^1B_{\infty,1}^{1/2}} &\leq \|v\|_{L^{2m}_{\omega}C_T^1C^{\eta}}\\
    \|q_0\|_{L_{\omega}^mC_tX}&\lesssim \|v\|_{L^{2m}_{\omega}C_TC^{\eta}}^2+\|v\|_{L^{m}_{\omega}C_T^1C^{\eta}}+C_z.
\end{align*}
Thus choose $C_{\rm start} = C\|v\|_{L^{2m}_{\omega}C_T^1C^{\eta}}$. In the induction step we define $g_{n+1}$ as before and let
\[g_{\leq n+1} := g_{\ell_{n+1}} + \chi g_{n+1} + g_{n+1}^{\rm pres}\]
where
\[\chi(t)=
\begin{cases}
    0, &t \in [0,2^{-n-1}]\cup[T-2^{-n-1},T],\\
    \in (0,1), &t \in (2^{-n-1},2^n) \cup(T-2^{-n}, T-2^{-n-1}),\\
    1, &t \in [2^{-n}, T-2^{-n}]
\end{cases}
\]
and
\[g_{n+1}^{\rm pres} := (P_{\lambda_{n+1}/3}-P_{\lambda_n/3})\Lambda^{-1}v.\]
Observe that
\[\|g_{n+1}^{\rm pres}\|_{L^{2m}C_TB_{\infty,1}^{1/2}} \lesssim \lambda_n^{-1} \|v\|_{L^{2m}_{\omega}C_T^1C^{\eta}} \leq \frac{1}{2}r_{n+1}^{\frac{1}{2}}.\]
Furthermore
\begin{align*}
    g_{\leq n+1}(0) &= P_{\lambda_{n+1}/3}\Lambda^{-1}v(0),\\
    g_{\leq n+1}(t) &= P_{\lambda_{n+1}/3}\Lambda^{-1}v(T), \quad\forall t \in [T-2^{-n-1},T]
\end{align*}
and a similar analysis as in \cite[Section 3.3]{hofmanovaconvex4} applies.
\end{proof}

\subsection{Coming down from infinity}\label{sec:comingInf}
Theorem \ref{thm:initCond} and Theorem \ref{thm:presTerminal} imply the following coming down from infinity result:
\begin{theorem}
    Let $T\geq 4, \epsilon >0$ and $\theta_0$ as in Theorem \ref{thm:initCond}, then there exist infinitely measurable solutions to \eqref{sqg intro} in the sense of Lemma \ref{lem:trafo} such that for all $m \geq 1$
    \begin{equation*}
        \|\theta\|_{L^{2m}_{\omega}C([T,\infty),B_{\infty,1}^{-1/2-})} \leq \epsilon.
    \end{equation*}
\end{theorem}
\begin{proof}
    Choose $v$ as in Theorem \ref{thm:presTerminal} such that $v(0) = \theta_0$ and $v(T)=0$, then from Theorem \ref{thm:presTerminal} we deduce the existence of $\theta$ such that
    \begin{align*}
        \theta(0) &=\theta_0,\\
        \theta(T) &= 0.
    \end{align*}
    Conclude with Theorem \ref{thm:initCond}.
\end{proof}

\subsection{Stationary (ergodic) solutions}\label{sec:statSol}
Stationary solutions satisfy the equation on the full time horizon $\mathbb{R}$ hence in order to study the existence of such solutions in the above setting, we must be able to extend the convex integration results to $\mathbb{R}$ as well. Hence in this subsection we can only treat the case of Section \ref{subsec:whiteNoise} which allows for the same extension procedure as described in \cite[Section 8]{hofmanovaconvex6}. Analogously we use Theorem \ref{thm:initCond} to start in $\theta_0 =0$ and apply the convex integration mechanism of Section \ref{convex int section} to
\[z_t = \int_{-\infty}^t e^{-(t-s)\Lambda^{\gamma}}\Lambda^{\delta-1}{\rm d}W_s\]
to construct infinitely many (ergodic) stationary solutions. Indeed: due to lack of uniqueness, we understand stationarity as invariance with respect to time shifts
\[S_t(\theta,W)(\cdot) := (\theta(\cdot + t), W(\cdot + t) - W(t)), \quad t \in \mathbb{R}, (\theta, W) \in \mathcal{T},\]
on the trajectory space $\mathcal{T} := C\left(\mathbb{R}, B_{p,1}^{-1/2}\right) \times C\left(\mathbb{R}, B_{p,p}^{-1-}\right)$ for some $p \in [2, \infty)$. 

In analogy to \cite[Definition 8.1]{hofmanovaconvex6} we hence define stationary solutions $((\Omega, \mathcal{F}, (\mathcal{F}_t),\mathbb{P}),\theta, W)$ as analytically weak solutions of \eqref{sqg intro} in the sense of Lemma \ref{lem:trafo} on $(-\infty, \infty)$ such that
\[{\rm Law}(S_t(\theta,W)) = {\rm Law}(\theta,W) \quad \text{for all }t\in \mathbb{R}.\]
Similarly we may define a stationary solution $((\Omega, \mathcal{F}, (\mathcal{F}_t),\mathbb{P}),\theta, W)$ as ergodic (cf. \cite[Definition 8.4]{hofmanovaconvex6}) if
\[ {\rm Law}(\theta, W)(A) \in \{0,1\}\quad \text{for all } A \subset \mathcal{T} \text{ Borel and shift invariant}.\]
In analogy to \cite[Theorem 8.2, Corollary 8.3, Theorem 8.5]{hofmanovaconvex6} and \cite[Theorem 4.2]{hofmanovaconvex4}) we may therefore deduce
\begin{theorem}
	\begin{enumerate}
		\item Each convex integration solution $\theta$ obtained in Theorem \ref{thm:initCond} with $\theta_0 =0$ and extended to $\mathbb{R}$ gives rise to a stationary solution $((\tilde{\Omega}, \tilde{\mathcal{F}}, (\tilde{\mathcal{F}}_t),\tilde{\mathbb{P}}),\tilde{\theta}, \tilde{W})$ such that time averages of the law of time-shifts of $\theta$ converge to the law of $(\tilde{\theta}, \tilde{W})$ and for some $\epsilon \in (0,1), \sigma >0$ it holds for all $m \geq 1$ $\tilde{\theta} \in L_{\omega}^{2m}C_tB_{\infty,1}^{-1/2+\epsilon} \cap L_{\omega}^{2m}C_t^{1/2-\sigma}B_{\infty,1}^{-1-\epsilon}$.
		\item There exist infinitely many non-Gaussian stationary solutions.
		\item There exist infinitely many non-Gaussian ergodic stationary solutions such that for some $\epsilon \in (0,1), \sigma >0$ and for all $m \geq 1$ they lie in $\tilde{\theta} \in L_{\omega}^{2m}C_tB_{\infty,1}^{-1/2+\epsilon} \cap L_{\omega}^{2m}C_t^{1/2-\sigma}B_{\infty,1}^{-1-\epsilon}$.
	\end{enumerate}
\end{theorem}
The proofs follow a tightness argument for the ergodic averages to deduce stationarity as well as a Krein-Milman argument to deduce ergodicity. Non-Gaussianity follows as described in the proof of Theorem \ref{thm:initCond}.

\appendix
\section{Auxiliary results}
\subsection{A multiplier and commutator Lemma}

\begin{lemma}[{\cite[Lemma~3.2]{MR4340931}}] \label{lem:LinfQuantified}
    Suppose $a \in L^\infty(\mathbb{T}^2)$ such that $a$ is mean-free with $\mathrm{supp}(\widehat{a}) \subset \{ \abs{k} \leq \mu \}$ and $\mu \geq 10$. Let $m \in C^\infty( \mathbb{R}^2 \backslash \{ 0 \} )$ be a homogeneous function of degree 0 and $T_m$ is the Fourier multiplier defined by $\widehat{T_m f}(k) = m(k) \widehat{f}(k)$. Then there exists a constant $C_m \geq 1$ such that
    \begin{align}
        \norm{T_m a}_{L^\infty(\mathbb{T}^2)} \leq C_m \log(\mu) \norm{a}_{L^\infty(\mathbb{T}^2)}.
    \end{align}
\end{lemma}

The next lemma allows to define the non-linearity in~\eqref{eq:SolutionConcept} for $g \in H^{1/2}$. Notice that for smooth $g$ and $\xi$ it holds
\begin{align*}
    \langle \nabla \cdot [ \nabla^{\perp} g \Lambda g ], \xi \rangle &= -\langle \Lambda g ,\nabla^{\perp} g  \cdot \nabla \xi \rangle.
\end{align*}

Moreover, $(g,g') \mapsto \langle \Lambda g ,\nabla^{\perp} g'  \cdot \nabla \xi \rangle$ for fixed $\xi \in H^4(\mathbb{T}^2)$ extends to a bi-linear and bounded map from $H^{1/2}(\mathbb{T}^2)\times H^{1/2}(\mathbb{T}^2) \to \mathbb{R}$. This is closely related to the commutator estimate of Cheng, Kwon and Li~\cite[Proposition~5.1]{MR4340931}. However, they work on the level $\theta = \Lambda g$, which is why for the sake of completeness, we provide the proof for our setting. 

\begin{lemma}\label{lem:Commutator}
There exists a constant $C \geq 1$ such that for all $g, g' \in H^{1/2}$ and $\xi \in H^4$ it holds
\begin{align}
    \abs{ \langle \Lambda g ,\nabla^{\perp} g'  \cdot \nabla \xi \rangle } \leq C \norm{g}_{\dot{H}^{1/2}} \norm{g'}_{\dot{H}^{1/2}} \norm{\xi}_{\dot{H}^4}.
\end{align}
\end{lemma}

\begin{proof}
Notice that 
    \begin{align*}
        \langle \Lambda g , \nabla^\perp g' \cdot \nabla \xi \rangle  &= \sum_{k,l}   \abs{k} l \cdot k^\perp  \widehat{g} (k) \widehat{g'}(l) \widehat{\xi}(k-l) \\
        &= \sum_{k \neq 0} \sum_{l \neq 0,k}   \abs{k} l \cdot k^\perp  \widehat{g} (k) \widehat{g'}(l) \widehat{\xi}(k-l).
    \end{align*}
H\"older's inequality shows
\begin{align*}
    &\sum_{k\neq 0} \sum_{l \neq 0,k}  \abs{k} l \cdot k^\perp  \widehat{g} (k) \widehat{g'}(l) \widehat{\xi}(k-l) \\
    &\hspace{2em} \leq  \left( \sum_{k\neq 0} \abs{k} \abs{ \widehat{g} (k) }^2 \right)^{1/2} \left( \sum_{k\neq 0} \abs{ \sum_{l \neq 0,k} \abs{k}^{1/2} l \cdot k^\perp  \widehat{g'}(l) \widehat{\xi}(k-l) }^2 \right)^{1/2}.
\end{align*}
Therefore it remains to verify
\begin{align} \label{eq:WantToShow}
     &\left( \sum_{k\neq 0} \abs{ \sum_{l \neq 0,k} \abs{k}^{1/2} l \cdot k^\perp  \widehat{g'}(l) \widehat{\xi}(k-l) }^2 \right)^{1/2} \lesssim \norm{g'}_{\dot{H}^{1/2}} \norm{\xi}_{\dot{H}^4}.
\end{align}
    
We distinguish two cases.
\begin{enumerate}
    \item \label{it:CaseBigDifferences} If $ \abs{l} \lesssim \abs{k-l}$, then $\abs{k} \lesssim \abs{k-l}$ and
\begin{align*}
  \abs{k}^{1/2} l \cdot k^\perp &\leq  \abs{k}^{3/2} \abs{l}  \lesssim \abs{k-l}^{2} \abs{l}^{1/2}.
\end{align*}
\item  \label{it:CaseSmallDifferences}  If $\abs{k-l} \ll \abs{l} $, then $\abs{k} \sim \abs{l}$ and  
\begin{align*}
  \abs{k}^{1/2} l \cdot k^\perp &= \abs{k}^{1/2} (k-l) \cdot k^\perp  \leq \abs{k}^{3/2} \abs{k-l}   < \abs{k}^{3/2} \abs{k-l}^{1/2} \abs{l}^{1/2}.
\end{align*}
\end{enumerate}

Split up the interior sum of the left hand side of~\eqref{eq:WantToShow} into the two cases and estimate separately. For fixed $k$ by H\"older's inequality
\begin{align*}
    &\sum_{l \neq 0,k \, \& \abs{l} \lesssim \abs{k-l}}  l \cdot k^\perp  \widehat{g'}(l) \widehat{\xi}(k-l) \\
    &\hspace{2em} \leq \sum_{l \neq 0,k \&  \abs{l} \lesssim \abs{k-l}}  \abs{l}^{1/2} \abs{\widehat{g'}(l)} \abs{k-l}^{2} \abs{\widehat{\xi}(k-l)} \\
    &\hspace{2em} \leq \left( \sum_{l \neq 0,k \&  \abs{l} \lesssim \abs{k-l}}  \abs{l} \abs{\widehat{g'}(l)}^2 \right)^{1/2} \left( \sum_{l \neq 0,k \&  \abs{l} \lesssim \abs{k-l}}  \abs{k-l}^{4} \abs{\widehat{\xi}(k-l)}^2 \right)^{1/2}.
\end{align*}

The second case can be treated analogously,
\begin{align*}
    &\sum_{l \neq 0,k \, \& \abs{k-l} \ll \abs{l} }   l \cdot k^\perp  \widehat{g'}(l) \widehat{\xi}(k-l)  \\
     &\hspace{2em} \leq \sum_{l \neq 0,k \, \& \abs{k-l} \ll \abs{l} }  \abs{l}^{1/2} \abs{\widehat{g'}(l)} \abs{k}^{3/2} \abs{k-l}^{1/2} \abs{\widehat{\xi}(k-l)} \\
    &\hspace{2em} \leq \left( \sum_{l \neq 0,k \, \& \abs{k-l} \ll \abs{l} }  \abs{l} \abs{\widehat{g'}(l)}^2 \right)^{1/2} \left( \sum_{l \neq 0,k \, \& \abs{k-l} \ll \abs{l} }  \abs{k}^3 \abs{k-l} \abs{\widehat{\xi}(k-l)}^2 \right)^{1/2}.
\end{align*}

Lastly, we need to bound
\begin{align*}
    \sum_{k\neq 0}  \sum_{l \neq 0,k \&  \abs{l} \lesssim \abs{k-l}}  \abs{k-l}^{4} \abs{\widehat{\xi}(k-l)}^2 \lesssim  \sum_{k\neq 0} \abs{k}^{8} \abs{\widehat{\xi}(k)}^2 = \norm{\xi}_{\dot{H}^{4}}^2.
\end{align*}

A change of variables $(k,l) \mapsto (k',l') = (k,k-l)$ reveals a similar bound for the second case. Indeed, $ \abs{l'} =  \abs{k-l} \ll \abs{l} \sim \abs{k} = \abs{k'}$ and, therefore,
\begin{align*}
    &\sum_{k\neq 0}  \sum_{l \neq 0,k \, \& \abs{k-l} \ll \abs{l} }  \abs{k}^3 \abs{k-l}  \abs{\widehat{\xi}(k-l)}^2 \\
    &\hspace{2em}=\sum_{k'\neq 0}  \sum_{l' \neq 0,k' \, \& \abs{l'} \ll \abs{k'} }  \abs{k'}^3 \abs{l'}  \abs{\widehat{\xi}(l')}^2 \\
    &\hspace{2em}\lesssim \sum_{k'\neq 0}  \sum_{l' \neq 0,k' \, \& \abs{l'} \ll \abs{k'} }  \abs{k'}^{-4} \abs{l'}^8  \abs{\widehat{\xi}(l')}^2 \lesssim \norm{\xi}_{\dot{H}^{4}}^2.
\end{align*}
\end{proof}

\subsection{Transformation of SQG equations}
\label{equation trafo intro}
We have the following 
\begin{lemma}\label{lem:trafo}
    Let $g\in L^{2m}C_{\rm loc}B^{1/2}_{\infty, 1}$ be a weak solution to \eqref{eq:SolutionConcept}. Let 
    \[
    z_t=\int_0^t e^{-\nu(t-s)\Lambda^\gamma}\Lambda^{\delta-1}{\rm d}W_s
    \]
    be well defined and enjoy the regularity \eqref{regularity z}. Then $\theta:=\Lambda(g+z)$ is a weak solution to \eqref{sqg intro} in the sense that 
    \[
    \langle  \theta_t, \xi\rangle-\langle  \theta_0, \xi\rangle +\int_0^t\langle \mathcal{R}^\perp \theta_s \cdot \nabla\theta_s, \xi \rangle ds=-\nu \int_0^t\langle \Lambda^\gamma\theta_s, \xi \rangle ds+\langle \Lambda^\delta (W_t-W_0), \xi\rangle
    \]
     for some $\theta_0\in H^{1/2}$ and any $\xi\in C^\infty_c(\mathbb{T}^2)$ where the non-linearity is understood by means of the identity
    \[
    \langle \mathcal{R}^\perp \theta \cdot \nabla \theta, \xi \rangle=\frac{1}{2}\langle \theta, [ \mathcal{R}^\perp\cdot, \nabla \xi] \theta\rangle
    \]
\end{lemma}
\begin{proof}
    As already mentioned prior to Lemma \ref{lem:Commutator} it was already shown in \cite[Proposition 5.1]{MR4340931} that the commutator $ [\mathcal{R}^\perp\cdot, \nabla \xi]$ extends to a linear map $[\mathcal{R}^\perp\cdot, \nabla \xi]: \dot{H}^{-1/2}\to \dot{H}^{1/2}$. Hence, $\theta:=\Lambda(g+z)$ enjoys sufficient regularity for the non-linearity to be well-defined. The claim thus follows by a direct calculation. 
\end{proof}

\subsection{The Sewing Lemma}
\label{sewing section}
We recall the Sewing Lemma due to \cite{gubi} (see also \cite[Lemma 4.2]{frizhairer}). Let $E$ be a Banach space, $[0,T]$ a given interval. Let $\Delta_n$ denote the $n$-th simplex of $[0,T]$, i.e. $\Delta_n:\{(t_1, \dots, t_n)| 0\leq t_1\dots\leq t_n\leq T \} $. For a function $A:\Delta_2\to E$ define the mapping $\delta A: \Delta_3\to E$ via
\[
(\delta A)_{s,u,t}:=A_{s,t}-A_{s,u}-A_{u,t}
\]
Provided $A_{t,t}=0$ we say that for $\alpha, \beta>0$ we have $A\in C^{\alpha, \beta}_2(E)$ if $\norm{A}_{\alpha, \beta}<\infty$ where
\begin{align*}
   \norm{A}_\alpha &:=\sup_{(s,t)\in \Delta_2}\frac{\norm{A_{s,t}}_E}{|t-s|^\alpha}, \qquad \norm{\delta A}_{\beta}:=\sup_{(s,u,t)\in \Delta_3}\frac{\norm{(\delta A)_{s,u,t}}_E}{|t-s|^\beta}, \\ \norm{A}_{\alpha, \beta} &:=\norm{A}_\alpha+\norm{\delta A}_\beta.
\end{align*}
For a function $f:[0,T]\to E$, we note $f_{s,t}:=f_t-f_s$

Moreover, if for any sequence $(\mathcal{P}^n([s,t]))_n$ of partitions of $[s,t]$ whose mesh size goes to zero, the quantity 
\[
\lim_{n\to \infty}\sum_{[u,v]\in \mathcal{P}^n([s,t])}A_{u,v}
\]
converges to the same limit, we note
\[
(\mathcal{I} A)_{s,t}:=\lim_{n\to \infty}\sum_{[u,v]\in \mathcal{P}^n([s,t])}A_{u,v}.
\]

\begin{lemma}[Sewing]
Let $0<\alpha\leq 1<\beta$. Then for any $A\in C^{\alpha, \beta}_2(E)$, $(\mathcal{I} A)$ is well defined. Moreover, denoting $(\mathcal{I} A)_t:=(\mathcal{I} A)_{0,t}$, we have $(\mathcal{I} A)\in C^\alpha([0,T], E)$ and $(\mathcal{I} A)_0=0$ and for some constant $c>0$ depending only on $\beta$ we have
\begin{equation}
    \norm{(\mathcal{I} A)_{t}-(\mathcal{I} A)_{s}-A_{s,t}}_{E}\leq c\norm{\delta A}_\beta |t-s|^\beta.
    \label{sewing a priori}
\end{equation}
We say the germ $A$ admits a sewing $(\mathcal{I}A)$ and call $\mathcal{I}$ the sewing operator. 
\label{sewing}
\end{lemma}

  \begin{lemma}['Schauder estimate']
        We have for $a\geq 0$ and $\theta\in [0,1]$ and $0\leq s\leq t$
        \[
        (e^{-at}-e^{-as})\lesssim  a^{\theta}(t-s)^\theta.
        \]
        \label{schauder}
    \end{lemma}
\begin{proof}
    The proof follows from interpolating the bounds
    \[
    |e^{-at}-e^{-as}|\leq 2, \qquad |e^{-at}-e^{-as}|=|\int_{as}^{at}-e^{-x}{\rm d}x|\leq a(t-s).
    \]
\end{proof}

\printbibliography 


\section*{Acknowledgements}
The authors thank Martina Hofmanová for several helpful discussions. Florian Bechtold and Theresa Lange have received funding from the European Research Council (ERC) under the European Union’s Horizon 2020 research and innovation programme (grant agreement No. 949981). J\"orn Wichmann was partially supported by the Australian Government through the Australian Research Council’s Discovery
Projects funding scheme (grant number DP220100937).
\end{document}